\documentclass[11pt]{amsart}
\usepackage{graphicx} % Required for inserting images
\usepackage[margin=1in]{geometry}

\usepackage{amsmath,amssymb,amsthm}
\usepackage{mathtools}
\usepackage[colorlinks=true, allcolors=blue]{hyperref}
\numberwithin{equation}{section} % Make equation numbering depend on section number

\usepackage{booktabs}

\IfFileExists{mathabx.sty}%
  {\DeclareFontFamily{U}{mathx}{\hyphenchar\font45}%
   \DeclareFontShape{U}{mathx}{m}{n}{<->mathx10}{}%
   \DeclareSymbolFont{mathx}{U}{mathx}{m}{n}%
   \DeclareFontSubstitution{U}{mathx}{m}{n}%
   \DeclareMathAccent{\widebar}{0}{mathx}{"73}%
}{%
  \PackageWarning{mathabx}{%
    Package mathabx not available, therefore\MessageBreak substituting
    widebar with overline\MessageBreak }%
  \newcommand{\widebar}[1]{\overline{#1}}%
}

\DeclarePairedDelimiter{\abs}{\lvert}{\rvert}

\DeclareMathOperator*{\argmax}{arg\,max}

\usepackage[most]{tcolorbox}

\usepackage[toc,page]{appendix}
\usepackage[ruled,vlined]{algorithm2e}
\usepackage[noend]{algpseudocode}

\usepackage{xcolor}

\usepackage{mathrsfs}
\DeclareFontFamily{U}{rsfs}{\skewchar\font127 }
\DeclareFontShape{U}{rsfs}{m}{n}{%
   <-6.5> rsfs5
   <6.5-8> rsfs7
   <8-> rsfs10
}{}

\newcommand{\dd}{\mathrm{d}}

\usepackage{bbm}

\newtheorem{theorem}{Theorem}[section]
\newtheorem{definition}{Definition}[section]
\newtheorem{assumption}{Assumption}[section]
\newtheorem{lemma}{Lemma}[section]
\newtheorem{corollary}{Corollary}[section]
\newtheorem{proposition}{Proposition}[section]

\theoremstyle{remark}
\newtheorem{remark}{Remark}[section]

\newtheorem*{notation}{Notation}

\title{State-space models through the lens of ensemble control}

\thanks{This research is supported in part by the National Science Foundation under awards DMS-2309378 and IIS-2403276. We thank Albert Gu and Andrej Risteski for helpful discussions on state-space models.}
\date{\today}
\author{Ye Feng}
\address{Mathematics Department, Duke University}
\email{ye.feng@duke.edu}

\author{Jianfeng Lu}
\address{Department of Mathematics, Department of Physics, Department of Chemistry, Duke University}
\email{jianfeng@math.duke.edu}

\begin{document}

\begin{abstract}
State-space models (SSMs) are effective architectures for sequential modeling, but
a rigorous theoretical understanding of their training dynamics is still lacking.
We formulate continuous-time SSM parameter-path optimization as an ensemble optimal
control problem: each input sequence generates a corresponding state trajectory in
the ensemble, while the trainable parameter path forms a common control shared by
all trajectories. Within this formulation, model evaluation is represented by the
forward state equation and backward sensitivity propagation by the corresponding
adjoint equation. We show that the Hamiltonian gradient with respect to the control
variable represents the negative first-variation density of the reduced ensemble
objective. Using this representation, we analyze a Bregman mirror-descent scheme
for the continuous-time ensemble objective; in Euclidean geometry this reduces to
functional projected gradient descent. State-adjoint and Hamiltonian-gradient
stability yield two-sided relative curvature of the reduced objective and relative
convexity under sufficient regularization. We further show that in the relatively
convex regime the optimal-control problem admits a minimizer, which is unique under
relative strong convexity. For SSMs, the explicit stability-generated defect
$\kappa_{\mathrm{stab}}^{\mathrm{SSM}}$ gives an $O(1/k)$ objective rate at
$\tau=\kappa_{\mathrm{stab}}^{\mathrm{SSM}}$ and geometric convergence to the
unique minimizer for $\tau>\kappa_{\mathrm{stab}}^{\mathrm{SSM}}$, under the
stated stepsize condition.
\end{abstract}

\maketitle

\section{Introduction}

State-space models (SSMs) are a classical and highly versatile mathematical framework for modeling dynamical systems, in which observable sequences are generated by the evolution of latent states governed by differential or difference equations. Originating in control systems engineering, SSMs have long served as a foundational tool in time series analysis. Their simple yet expressive formulation, consisting of a latent state dynamics equation coupled with an observation model, has enabled widespread adoption across disciplines including economics \cite{zeng2013state}, neuroscience \cite{linderman2017bayesian,zoltowski2020general}, ecology \cite{glennie2023hidden,newman2023state}, robotics \cite{liu2024robomamba, liu2024point}, and weather forecasting \cite{yang2025wssm}.

In recent years, SSMs have re-emerged as powerful architectures for sequence modeling in machine learning, driven by advances that integrate neural parameterizations with structured dynamical systems. Modern SSM-based models demonstrate strong empirical performance and excellent scalability on long sequences, making them compelling alternatives to attention-based architectures. 
Structured state-space models such as S4 \cite{gu2021efficiently} achieve efficient long-sequence modeling through structured state dynamics. Building on this line, the Mamba architecture \cite{gu2024mamba} makes the SSM parameters input-dependent, enabling selective information propagation along the sequence while retaining linear computational complexity.
More recently, Mamba-3 \cite{lahoti2026mamba} further enhances expressivity through complex-valued state models and higher-order discretization, and improves inference efficiency via a multi-input multi-output (MIMO) design. Compared to transformers, which rely on pairwise attention with quadratic complexity in sequence length, SSM-based models represent sequences through dynamical systems with structured evolution laws, providing a principled and efficient mechanism for long-range dependency modeling and an interpretable state representation.

Although compelling empirical evidence confirms the remarkable performance of SSM-based architectures, there remains a pressing need for a rigorous mathematical framework to analyze their fundamental properties, particularly the mechanisms underlying their training and learning dynamics. In the existing literature, the analogy between deep learning and optimal control has become a well-established and fruitful viewpoint. In this framework, the forward propagation of features is interpreted as the trajectory of a dynamical system governed by parameterized equations, while the learning process corresponds to solving an optimal control problem that minimizes a cost functional defined by the training objective.

This control-theoretic interpretation has been developed for several major architectures, including residual networks \cite{han2019mean,li2018optimal}, neural ordinary differential equations (neural ODEs) \cite{li2018maximum,ruiz2023neural}, and more recently, transformers \cite{kan2025optimal,zhang2024optimal}. In these works, the state dynamics in the control problem correspond to the forward pass through the layers of a neural network
\[
    \dot{x}(t) = b(t, x(t), u(t)), \quad x(0) \sim \mu_0,
\]
starting from an initial distribution $\mu_0$ of data inputs. The control variables correspond to the network parameters or weights that influence the state evolution.
This framework has proven useful for understanding the training of deep neural networks and continuous-depth models at scale. Within this setting, Pontryagin’s maximum principle (PMP) provides a continuous-time adjoint framework closely related to backpropagation: the adjoint states act as Lagrange multipliers enforcing the dynamic constraints and propagate sensitivity information backward through the dynamics.
This duality between control and learning has inspired a wide range of theoretical analyses, including continuous-depth and mean-field formulations, as well as the design of training algorithms grounded in optimal control principles.

\medskip

The main motivation for the current study is to formulate continuous-time SSM parameter-path optimization within the control-theoretic viewpoint. For selective and other input-dependent SSMs, the state evolution depends on the input sequence in addition to the current latent state and the trainable parameters. It is therefore natural to write
\[
\dot x_\omega(t) = b(t,x_\omega(t),u(t),\omega), 
\qquad x_\omega(0) = x_0(\omega),
\]
where $\omega$ denotes the input signal (see Table~\ref{tab:ssm_mapping} for a summary of the SSM notation). 
Here the control $u$ parameterizes the trainable coefficients of the continuous-time model. The resulting ensemble-control problem provides an ambient functional optimization formulation for the continuous-time parameter path. Standard SSM parameterizations with shared finite-dimensional parameters arise as restrictions of this formulation. Their gradients are obtained from the functional-gradient density through the accumulation and chain-rule formulas in Remark~\ref{rem:tied-parameter-training}. Within this framework, we analyze Bregman mirror descent for the continuous-time reduced objective.

In an ensemble control problem, a single control simultaneously governs a family of systems indexed by a parameter $\omega \in \Omega$. In the present setting, the ensemble index is the SSM input sequence, and the common control consists of the parameters shared across those inputs. The cost functional aggregates the training loss over the input distribution, just as the original training objective evaluates one shared model across multiple sequences.

This identification is the central viewpoint of the paper. Different input sequences index different state trajectories; the shared trainable parameters form the common control; forward model evaluation is the evolution of the state equation; and backward sensitivity propagation is represented by the adjoint equation. The ensemble formulation therefore places the forward and backward computations of continuous-time training in a single analytical framework and supports a quantitative analysis of functional mirror descent through stability, relative smoothness, energy dissipation, and relative convexity.

Our main contributions are threefold. First, we formulate continuous-time SSM
parameter-path optimization as an ensemble optimal control problem and derive its
state-adjoint system and adjoint representation of the first variation. We also
make explicit how the resulting functional-gradient density reduces to the usual
gradient with respect to tied finite-dimensional parameters. Second, for a general
class of ensemble control problems, we show that state-adjoint and
Hamiltonian-gradient stability yield two-sided relative curvature of the reduced objective, leading to energy dissipation, existence and uniqueness of minimizers, and convergence rates for Bregman mirror descent. Third, specializing these estimates to SSMs gives explicit stability constants and rank-free convergence guarantees for functional projected gradient descent on the continuous-time parameter-path objective, with an $O(1/k)$ objective rate at the critical regularization threshold and geometric convergence to the unique minimizer above it.

Sections~\ref{sec:form}--\ref{sec:adjoint-mirror} introduce ensemble control and Bregman mirror descent. Section~\ref{sec:ensemble-mirror-convergence} proceeds from state-adjoint stability to Hamiltonian-gradient stability and two-sided relative curvature of the reduced objective. The upper curvature bound yields relative smoothness and energy dissipation, while the lower bound yields relative convexity, existence of minimizers, uniqueness under relative strong convexity, and convergence rates under sufficient regularization. Section~\ref{sec:ssm-consequences} specializes these results to the SSM parameter-path objective, with explicit constants and projected-gradient convergence estimates that require no rank assumption on the readout.

\subsection*{Related Work}

Recent studies have increasingly applied optimal control theory to understand, train, and enhance deep learning models. Li and Hao \cite{li2018optimal} formulate neural network training as a discrete-time optimal control problem and introduce a Pontryagin-based method of successive approximations that enables gradient-free and stable optimization, even for networks with discrete or ternary weights. Extending this formulation, Li et al. \cite{li2018maximum} treat deep learning in continuous time and employ Pontryagin's maximum principle (PMP) to derive an alternative training algorithm with provable convergence that alleviates common issues of gradient-based methods such as slow convergence and saddle-point trapping. E et al. \cite{han2019mean} generalize this perspective to population risk minimization by introducing a mean-field optimal control framework governed by the Hamilton-Jacobi-Bellman and Pontryagin principles, thereby establishing a rigorous mathematical link between optimal control and learning dynamics. More recently, Ruiz-Balet and Zuazua \cite{ruiz2023neural} analyze neural ordinary differential equations (neural ODEs) through a control-theoretic lens, demonstrating that data classification and universal approximation can be achieved by simultaneously controlling multiple neural ODEs with a shared control law. Furthermore, Chen et al. \cite{chen2022self} exploit a closed-loop control formulation to construct self-healing neural networks capable of autonomously detecting and correcting adversarial perturbations, thus improving robustness.

Recent advances have extended these control-theoretic insights to modern architectures. Zhang et al. \cite{zhang2024optimal} cast vision transformer (ViT) fine-tuning as an optimal control problem solved via the PMP, revealing that low-precision (binary) adaptation matrices can preserve the controllability of their full-precision counterparts. Kan et al. \cite{kan2025optimal} apply optimal control theory to both training and architecture design of transformers, proposing a continuous-time framework that improves performance, generalization, and parameter efficiency while offering theoretical grounding for model design and training improvements.

Existence questions for ensemble optimal control have also been studied in more
classical control settings. Scagliotti~\cite{scagliotti2023optimal} considers
ensembles of affine-control systems and establishes well-posedness of the
minimization problem together with approximation results for finite
sub-ensembles. Bartsch, Borz\`i, Fanelli, and Roy~\cite{bartsch2019brockett}
study ensemble control through a Liouville-equation formulation and prove
existence of optimal controls, as well as uniqueness under additional
assumptions. These results concern structural settings different from the
general nonlinear parameter dependence allowed here. In the relatively convex
regime used for our convergence theory, we instead obtain existence directly
from the curvature estimate and weak compactness of the admissible control set.

Starting from a similar control-theoretic viewpoint, our work focuses specifically on state-space models (SSMs). While earlier methods apply optimal control ideas to feedforward networks, neural ODEs, or transformers, we develop an optimal control formulation tailored to the dynamics of SSMs. 
Closely related to our work, recent theoretical studies on selective state-space models, such as \cite{muca2024theoretical}, formulate input-dependent SSMs as controlled differential equations and analyze their expressive power using tools from rough path theory, establishing universality and representation guarantees. Our focus is instead on the shared-parameter training problem across a family of input-indexed trajectories. By viewing these trajectories as an ensemble governed by a common trainable control, we obtain an adjoint first-variation representation and quantitative convergence results for continuous-time SSM parameter-path optimization.

\begin{notation}
    For a vector $x\in\mathbb R^n$, $|x|$ denotes its Euclidean norm; for scalars, $|\cdot|$ denotes absolute value. Matrices, linear maps, and multilinear maps are measured by the induced operator norm $\|\cdot\|_{\mathrm{op}}$. In particular, for a $k$-linear map $T$ with values in a Euclidean space,
    \[
        \|T\|_{\mathrm{op}}
        :=\sup_{|z_1|=\cdots=|z_k|=1}|T[z_1,\ldots,z_k]|.
    \]
    For matrix-valued multilinear maps, the output is measured in the matrix operator norm; equivalently, one takes the supremum of $|T[z_1,\ldots,z_k]v|$ over unit $z_1,\ldots,z_k,v$. Function supremum norms are denoted explicitly by $\|\cdot\|_\infty$.

    Let $(\Omega, \Sigma, \mu)$ be a probability space. For a measurable function $z:[0,T]\times\Omega\to\mathbb R^n$ and a measurable function $u:[0,T] \to \mathbb{R}^m$, we denote their $L^2$ norms by
    \[
    \|z\|_{L^2}:=\left(\int_0^T\mathbb E_\omega|z(t,\omega)|^2\,\dd t\right)^{1/2},
    \qquad
    \|u\|_{L^2}:=\left(\int_0^T|u(t)|^2\,\dd t\right)^{1/2}.
    \]
\end{notation}

\section{SSM Training as Ensemble Control} \label{sec:form}

We first formulate continuous-time SSM training as optimization over a parameter path shared by all input-dependent trajectories. We then exhibit this structure in S4 and selective SSMs, before passing to the general ensemble-control framework used in the analysis.

\subsection{SSMs and continuous-time training formulation}\label{subsec:ssm_training_formulation}

We model input and output sequences as continuous-time trajectories on a fixed interval $[0,T]$. The target mapping
\[
    \mathcal F:\Omega\to\Omega',\qquad \omega\mapsto\mathcal F[\omega],
\]
associates each input trajectory $\omega:[0,T]\to\mathbb R^d$ with an output trajectory $\mathcal F[\omega]:[0,T]\to\mathbb R^{d'}$. 
The input set $\Omega$ is a Borel subset of $C([0,T];\mathbb R^d)$, while $\Omega'$ is a subset of $C([0,T];\mathbb R^{d'})$, both equipped with the corresponding supremum-norm topology.
We equip $\Omega$ with its Borel $\sigma$-algebra $\Sigma$ and a fixed probability measure $\mu$; expectations below are taken with respect to this measure.

A measurable parameter path $u:[0,T]\to U\subset\mathbb R^m$ serves as the control of the continuous-time model: it is the trainable parameter path shared across all input trajectories. For each input $\omega$, it determines a latent state $x^u(\cdot,\omega):[0,T]\to\mathbb R^n$ through
\begin{align}\label{eq:ssm_dyn}
    \begin{cases}
        \dot x^u(t,\omega)
        =A(t,\omega;u(t))x^u(t,\omega)+B(t,\omega;u(t)),
        &t\in[0,T],\\[4pt]
        x^u(0,\omega)=x_0(\omega).
    \end{cases}
\end{align}
Here $A(t,\omega;\theta)\in\mathbb R^{n\times n}$ and $B(t,\omega;\theta)\in\mathbb R^n$ depend on a pointwise parameter vector $\theta\in U$.  
Optimizing over \(u\) defines the ambient functional optimization problem studied in this paper. The tied-parameter architectures displayed below are recovered by restricting the admissible parameter paths to finite-dimensional families, including the special case \(u(t)\equiv\theta\). The accumulation and chain-rule formulas in Remark~\ref{rem:tied-parameter-training} express the gradients of these shared finite-dimensional parameterizations in terms of the functional-gradient density.

The output is obtained by a fixed linear readout,
\[
    y^u(t,\omega)=Cx^u(t,\omega),\qquad C\in\mathbb R^{d'\times n}.
\]
An input-dependent, trainable readout can also be considered, but the present formulation uses a fixed matrix $C$. We define the admissible controls by
\[
    \mathcal U[0,T]
    :=\{u:[0,T]\to U\mid u\text{ is measurable}\}.
\]
For $u\in\mathcal U[0,T]$, the training objective is
\[
    J^\tau[u]
    :=\int_0^T\mathbb E_\omega
    \abs{Cx^u(t,\omega)-\mathcal F[\omega](t)}^2\,\mathrm dt
    +\frac{\tau}{2}\int_0^T\abs{u(t)}^2\,\mathrm dt,
    \qquad \tau\ge0.
\]
The first term measures agreement with the target, while the second penalizes parameter magnitude and supplies regularizing curvature. The SSM training problem is
\[
    (\mathbf P_S^\tau)\qquad
    \inf_{u\in\mathcal U[0,T]}J^\tau[u].
\]
An admissible control $u^\star$ is optimal if
\[
    J^\tau[u^\star]=\inf_{u\in\mathcal U[0,T]}J^\tau[u].
\]

\begin{remark}
    The coefficient of the data-fitting term in $J^\tau$ is normalized to
one without loss of generality. Indeed, for any $c>0$ and $\rho\ge 0$,
the objectives
\[
c\int_0^T \mathbb E_\omega
\abs{Cx^u(t,\omega)-\mathcal F[\omega](t)}^2\,\mathrm dt
+\frac{\rho}{2}\int_0^T\abs{u(t)}^2\,\mathrm dt
\]
and
\[
\int_0^T \mathbb E_\omega
\abs{Cx^u(t,\omega)-\mathcal F[\omega](t)}^2\,\mathrm dt
+\frac{\rho}{2c}\int_0^T\abs{u(t)}^2\,\mathrm dt
\]
have the same minimizers. Thus only the relative weight of the
regularization matters, and it is represented here by $\tau$.
\end{remark}

We impose the following structural conditions on the SSM training problem.
\begin{assumption}\label{asp:main_asp}
    For problem $(\mathbf P_S^\tau)$, assume:
    \begin{enumerate}
        \item The input set $\Omega$ is a compact Borel subset of
        $C([0,T];\mathbb R^d)$, $\Sigma$ is its Borel $\sigma$-algebra,
        and $\mu$ is a fixed probability measure on $(\Omega,\Sigma)$.
        \item The output set $\Omega'$ is a subset of
        $C([0,T];\mathbb R^{d'})$.
        \item The input--output map $\mathcal F:\Omega\to\Omega'$ and
        initial condition $x_0:\Omega\to\mathbb R^n$ are continuous.
        In particular,
        \[
            M_{\mathcal F}:=\sup_{\omega\in\Omega}
            \|\mathcal F[\omega]\|_\infty<\infty.
        \]
        \item The set $U$ is a nonempty compact convex subset of $\mathbb R^m$ with nonempty interior.
        \item The maps
        $A:[0,T]\times\Omega\times U\to\mathbb R^{n\times n}$ and
        $B:[0,T]\times\Omega\times U\to\mathbb R^n$ are jointly
        continuous and satisfy
        \[
            \|A(t,\omega;\theta)\|_{\mathrm{op}}\le M_A,\qquad
            \abs{B(t,\omega;\theta)}\le M_B
        \]
        uniformly on their domains.
        \item There is an open neighborhood $\mathcal O_U\subset\mathbb R^m$
        of $U$ such that, for each $(t,\omega)$, the maps
        $\theta\mapsto A(t,\omega;\theta)$ and
        $\theta\mapsto B(t,\omega;\theta)$ admit twice continuously
        differentiable extensions to $\mathcal O_U$. Moreover,
        \begin{align*}
            G_A&:=\sup_{(t,\omega,\theta)\in[0,T]\times\Omega\times U}
            \|D_\theta A(t,\omega;\theta)\|_{\mathrm{op}}<\infty,\\
            G_B&:=\sup_{(t,\omega,\theta)\in[0,T]\times\Omega\times U}
            \|D_\theta B(t,\omega;\theta)\|_{\mathrm{op}}<\infty,\\
            H_A&:=\sup_{(t,\omega,\theta)\in[0,T]\times\Omega\times U}
            \|D_\theta^2 A(t,\omega;\theta)\|_{\mathrm{op}}<\infty,\\
            H_B&:=\sup_{(t,\omega,\theta)\in[0,T]\times\Omega\times U}
            \|D_\theta^2 B(t,\omega;\theta)\|_{\mathrm{op}}<\infty.
        \end{align*}
    \end{enumerate}
\end{assumption}

\begin{remark}\label{rmk:op_norm_AB}
The derivative bounds in Assumption~\ref{asp:main_asp}(6) use the global operator-norm convention. They can be estimated from coordinate derivatives: if, uniformly in $(t,\omega,\theta)$,
\[
    \left\|\frac{\partial A(t,\omega;\theta)}{\partial\theta_i}\right\|_{\mathrm{op}}
    \le\bar G_A,\qquad
    \left\|\frac{\partial^2A(t,\omega;\theta)}{\partial\theta_i\partial\theta_j}\right\|_{\mathrm{op}}
    \le\bar H_A,\qquad 1\le i,j\le m,
\]
then the triangle inequality and Cauchy--Schwarz give
\begin{align*}
    \|D_\theta A(t,\omega;\theta)[z]\|_{\mathrm{op}}
    &\le\bar G_A\sum_i|z_i|
    \le\sqrt m\,\bar G_A|z|,\\
    \|D_\theta^2A(t,\omega;\theta)[z,z']\|_{\mathrm{op}}
    &\le\bar H_A\Bigl(\sum_i|z_i|\Bigr)\Bigl(\sum_j|z'_j|\Bigr)
    \le m\bar H_A|z|\,|z'|.
\end{align*}
Consequently $G_A\le\sqrt m\,\bar G_A$ and $H_A\le m\bar H_A$. The corresponding estimates for $G_B$ and $H_B$ follow from Euclidean bounds on the vector-valued coordinate derivatives of $B$ in the same way.
\end{remark}

\subsection{Examples: S4 and selective SSMs}\label{subsec:ssm_examples}

The coefficients in the SSM training formulation encompass standard and selective state-space architectures. We illustrate this connection through representative S4 and S6 dynamics. For clarity, the architecture formulas below display a fixed parameter vector $\theta$, corresponding to the special case $u(t)\equiv\theta$. Here $t\in[0,T]$ is the continuous sequence/state-evolution coordinate, not optimization time; optimization iterates below are indexed by $k$. The general formulation permits the shared parameter path $u(t)$ to vary along the sequence coordinate.

In practice, one often considers coefficient maps of the form
\[
    A(t, \omega; \theta) = A(\omega(t); \theta), \quad B(t, \omega; \theta) = B(\omega(t); \theta),
\]
so that the matrix $A$ and vector $B$ depend only on the value $\omega(t)$ at sequence coordinate $t$, rather than on the entire input trajectory. In this setting, the mappings $A$ and $B$ are jointly continuous in $(t, \omega, \theta)$ provided that the functions $(\xi, \theta) \mapsto A(\xi; \theta)$ and $(\xi, \theta) \mapsto B(\xi; \theta)$ are continuous.

Let us survey several representative examples arising in practice. A basic and widely used model is the standard linear state-space system
\begin{equation}\label{eq:linear_ssm}
    \dot{x}^u(t, \omega) = \mathbb{A}\,x^u(t, \omega) + \mathbb{B}\,\omega(t),
\end{equation}
where $\mathbb{A} \in \mathbb{R}^{n \times n}$ and $\mathbb{B} \in \mathbb{R}^{n \times d}$. The input signal is $\omega = (\omega_i)_{1 \leq i \leq d}$, where $\omega_i$ denotes the $i$-th input channel. If we regard $\mathbb{A}$ and $\mathbb{B}$ as trainable parameters, i.e., as the control variables of the model, then \eqref{eq:linear_ssm} clearly fits into the structural framework described in Assumption~\ref{asp:main_asp}.

In many structured state-space models (SSMs), the dynamics operate independently across input channels. A canonical example is the structured state-space sequence model (S4) \cite{gu2021efficiently,muca2024theoretical}, which can be viewed as a special case of \eqref{eq:linear_ssm} with both $\mathbb{A}$ and $\mathbb{B}$ block-diagonal. More precisely, let $n = Nd$ for some $N \in \mathbb{N}$ and write
\[
    \mathbb{A} =
    \begin{pmatrix}
        A_1 &     &        &     \\
            & A_2 &        &     \\
            &     & \ddots &     \\
            &     &        & A_d
    \end{pmatrix}
    \in \mathbb{R}^{Nd \times Nd}, \qquad
    \mathbb{B} =
    \begin{pmatrix}
        \tilde{B} &   &        &   \\
          & \tilde{B} &        &   \\
          &   & \ddots &   \\
          &   &        & \tilde{B}
    \end{pmatrix}
    \in \mathbb{R}^{Nd \times d},
\]
where each $A_i \in \mathbb{R}^{N \times N}$ is diagonal and $\tilde{B} \in \mathbb{R}^{N \times 1}$. In this case, the system decomposes into $d$ independent subsystems,
\begin{equation}\label{eq:diag_linear_ssm}
    \dot{x}_i^u(t, \omega) = A_i x_i^u(t, \omega) + \tilde{B} \omega_i(t),
    \qquad 1 \leq i \leq d,
\end{equation}
so that $x_i^u(t,\omega)$ depends only on the $i$-th channel $\omega_i$, i.e., $x_i^u(t,\omega) = x_i^u(t,\omega_i)$.

In practice, SSMs are implemented in discrete time via a suitable discretization of \eqref{eq:diag_linear_ssm}. For instance, under a zero-order hold (ZOH) scheme \cite[Chapter~2]{aastrom2013computer} with channel-dependent step size $\Delta_i$, one obtains the recurrence
\begin{equation}\label{eq:s4_rec}
    x_{i,l}^u = \bar{A}_i\,x_{i,l-1}^u + \bar{B}_i \omega_{i,l},
    \qquad 1 \leq l \leq L,\; 1 \leq i \leq d,
\end{equation}
where $\bar{A}_i = \exp(\Delta_i A_i)$ and $\bar{B}_i$ is determined accordingly by the discretization formula. Here, the gates $\bar{A}_i$ and $\bar{B}_i$ are channel-specific but independent of the sequence index $l$. The selective state-space model (S6) \cite{gu2024mamba,muca2024theoretical} extends (S4) by allowing the discretization to depend on the current input. Specifically, the recurrence becomes
\begin{equation}\label{eq:s6_rec}
    x_{i,l}^u =  \bar{A}_i(\omega_l)\, x_{i,l-1}^u + \bar{B}_i(\omega_l)\, \omega_{i,l},
    \qquad 1 \leq l \leq L,\; 1 \leq i \leq d,
\end{equation}
where $\omega_l = (\omega_{1,l}, \dots, \omega_{d,l})$ denotes the full input vector at sequence position $l$. Unlike \eqref{eq:s4_rec}, the gates $\bar{A}_i(\omega_l)$ and $\bar{B}_i(\omega_l)$ now depend on all input channels at that position. This input dependence is typically introduced through an input-controlled step size
\[
    \Delta_i: \mathbb{R}^d \to \mathbb{R},
    \qquad
    \Delta_i(y) = \operatorname{softplus}(q_i \cdot y + r_i),
\]
with trainable parameters $q_i \in \mathbb{R}^d$ and $r_i \in \mathbb{R}$. The discrete gates are then defined by
\[
    \bar{A}_i(y) = \exp(\Delta_i(y) A_i),
    \qquad
    \bar{B}_i(y) = (B\, y)\,\Delta_i(y),
\]
where each $A_i$ is diagonal and $B \in \mathbb{R}^{N \times d}$ is a shared linear projection across channels.

Motivated by the continuous-time interpretation in \cite{muca2024theoretical}, introduce a small step parameter $\delta>0$ and consider the scaled input-dependent step size
$$
    \Delta_i(y)
    =
    \delta\,\sigma(q_i\cdot y+r_i),
    \qquad
    \sigma(z):=\operatorname{softplus}(z).
$$
Then
$$
    \overline A_i(y)
    =
    \exp\!\bigl(\delta\,\sigma(q_i\cdot y+r_i)A_i\bigr)
    =
    I+\delta\,\sigma(q_i\cdot y+r_i)A_i+O(\delta^2),
$$
while
$$
    \overline B_i(y)
    =
    \delta\,(By)\,\sigma(q_i\cdot y+r_i).
$$
Consequently, to first order in $\delta$, the recurrence \eqref{eq:s6_rec} becomes
$$
    x_{i,l}^u
    =
    x_{i,l-1}^u
    +
    \delta\,A_i x_{i,l-1}^u
    \sigma(q_i\cdot\omega_l+r_i)
    +
    \delta\,(B\omega_l)\omega_{i,l}
    \sigma(q_i\cdot\omega_l+r_i)
    +
    O(\delta^2).
$$
This is the forward Euler discretization, to first order in $\delta$, of the controlled ODE
\begin{equation}\label{eq:cont_limit_s6}
\dot{x}_i^u(t,\omega)
=
A_i x_i^u(t,\omega)
\sigma(q_i\cdot\omega(t)+r_i)
+
(B\omega(t))\omega_i(t)
\sigma(q_i\cdot\omega(t)+r_i).
\end{equation}

In this formulation, the trainable parameters are $\{A_i\}_{i=1}^d$, $B$, and $\{q_i,r_i\}_{i=1}^d$, so the parameter vector $\theta$ can be identified with the vectorization of all these parameters. The continuous-time dynamics \eqref{eq:cont_limit_s6} therefore take the general form
\[
    \dot{x}^u(t, \omega) = A(\omega(t);\theta)\,x^u(t, \omega) + B(\omega(t);\theta),
\]
with a block-diagonal gate
\[
    A(\omega(t);\theta)
    =
    \begin{pmatrix}
        A_1\, \sigma({q}_1 \cdot \omega(t) + {r}_1)
         &                                                      &        &                                                      \\
         & A_2\, \sigma({q}_2 \cdot \omega(t) + {r}_2)
         &                                                      &                                                               \\
         &                                                      & \ddots &                                                      \\
         &                                                      &        & A_d\, \sigma({q}_d \cdot \omega(t) + {r}_d)
    \end{pmatrix},
\]
and a channel-wise gated projection
\[
    B(\omega(t);\theta)
    =
    \begin{pmatrix}
        (B\,\omega(t))\,\omega_1(t)\,
        \sigma({q}_1 \cdot \omega(t) + {r}_1)
        \\[0.4em]
        (B\,\omega(t))\,\omega_2(t)\,
        \sigma({q}_2 \cdot \omega(t) + {r}_2)
        \\
        \vdots
        \\
        (B\,\omega(t))\,\omega_d(t)\,
        \sigma({q}_d \cdot \omega(t) + {r}_d)
    \end{pmatrix}.
\]
The resulting coefficients $A(\omega(t);\theta)$ and $B(\omega(t);\theta)$ satisfy the regularity and structural conditions required in Assumption~\ref{asp:main_asp}; the required uniform derivative bounds follow from smoothness of the softplus activation and compactness of the input and parameter sets.

Table~\ref{tab:ssm_mapping} summarizes the SSM notation used throughout the paper.

\begin{table}[ht]
\centering
\caption{Summary of symbols in the continuous-time SSM formulation.}
\label{tab:ssm_mapping}
\begin{tabular}{@{}lll@{}}
\toprule
\textbf{Quantity} & \textbf{Symbol} & \textbf{Meaning} \\
\midrule
Input sequence & $\omega(t)$ & input at sequence coordinate $t$ \\
Hidden state & $x^u(t,\omega)$ & state associated with input $\omega$ \\
State dynamics & $\dot x^u(t,\omega)$ & continuous state evolution \\
SSM coefficients & $A,\,B$ & state matrix and forcing vector \\
Parameter path (control) & $u(t)$ & shared trainable parameters \\
Model output & $y^u(t,\omega)=Cx^u(t,\omega)$ & linear readout of the hidden state \\
\bottomrule
\end{tabular}
\end{table}

The examples above share a common structure: each input trajectory $\omega$ induces its own state trajectory, while a single parameter path $u$ is shared across all inputs. We refer to this shared-control family as an ensemble-control formulation and now abstract it into the general problem used in the analysis below.

\subsection{General ensemble optimal control formulation}\label{subsec:general_ensemble_control}

We formulate the corresponding general ensemble optimal control problem, allowing nonlinear state dynamics and general running and terminal costs. Let $(\Omega,\Sigma,\mu)$ be a probability space, with $\omega$ indexing an ensemble member. For a common measurable control $u:[0,T] \to U$, its state solves
\begin{align}\label{eq:ctrl_sys}
    \begin{cases}
        \dot x^u(t,\omega)=b(t,x^u(t,\omega),u(t),\omega),
        &t\in[0,T],\\[4pt]
        x^u(0,\omega)=x_0(\omega).
    \end{cases}
\end{align}
We define the training objective as follows, while separating the regularizer from the other costs:
\begin{align}\label{eq:cost_func_gen}
    J^\tau[u]&:=J^0[u]+\tau\mathcal R_h[u],\qquad \tau\ge0,\\
    J^0[u]&:=\int_0^T\mathbb E_\omega
    f(t,x^u(t,\omega),u(t),\omega)\,\mathrm dt
    +\mathbb E_\omega g(x^u(T,\omega),\omega),\nonumber\\
    \mathcal R_h[u]&:=\int_0^T h(u(t))\,\mathrm dt.\nonumber
\end{align}
Here $f$ is the unregularized running cost, $g$ is the terminal cost, and $h$ is a convex regularizer. The general ensemble control problem is
\[
    (\mathbf P^\tau)\qquad
    \inf_{u\in\mathcal U[0,T]}J^\tau[u],
\]
where $\mathcal U[0,T]=\{u:[0,T]\to U\mid u\text{ is measurable}\}$ is the admissible control set.
An admissible control $u^\star$ is optimal when
$J^\tau[u^\star]=\inf_{u\in\mathcal U[0,T]}J^\tau[u]$.
For bounded measurable functions $\xi,\pi:\Omega\to\mathbb R^n$ and a control (parameter) vector $\theta\in U$, define the Hamiltonians
\begin{align}
\label{eq:general-Hamiltonian-unregularized}
    H_t^0(\xi,\pi,\theta)
    &:=\mathbb E_\omega\!\left[
    \pi(\omega)\cdot b(t,\xi(\omega),\theta,\omega)
    -f(t,\xi(\omega),\theta,\omega)\right],\\
\label{eq:general-Hamiltonian-regularized}
    H_t^\tau(\xi,\pi,\theta)
    &:=H_t^0(\xi,\pi,\theta)-\tau h(\theta).
\end{align}

\begin{assumption}[Regularity of the general ensemble problem]
\label{ass:general-first-variation}
The following conditions hold.
\begin{enumerate}
    \item $U$ is a nonempty compact convex subset of $\mathbb R^m$ with nonempty interior.
    \item $x_0:\Omega\to\mathbb R^n$ is measurable and uniformly bounded;
    write $\|x_0\|_\infty:=\sup_{\omega\in\Omega}|x_0(\omega)|$.
    \item For every $(x,\theta)$${}\in\mathbb R^n\times U$, the maps
    \[
        (t,\omega)\mapsto b(t,x,\theta,\omega),\qquad
        (t,\omega)\mapsto f(t,x,\theta,\omega)
    \]
    are measurable. There is a constant $M_b>0$ such that
    \[
        |b(t,x,\theta,\omega)|\le M_b(1+|x|),
        \qquad x\in\mathbb R^n,\quad\theta\in U.
    \]
    \item There is an open neighborhood $\mathcal O_U\subset\mathbb R^m$
    of $U$ such that, for each $(t,\omega)$, the maps
    $(x,\theta)\mapsto b(t,x,\theta,\omega)$ and
    $(x,\theta)\mapsto f(t,x,\theta,\omega)$ are twice continuously
    differentiable on $\mathbb R^n\times\mathcal O_U$.
    For every $R>0$ there is $M_R>0$, independent of $(t,\omega)$, such that
    for $|x|\le R$ and $\theta\in U$,
    \[
        \begin{aligned}
        &\|D_x b(t,x,\theta,\omega)\|_{\mathrm{op}}
        +\|D_\theta b(t,x,\theta,\omega)\|_{\mathrm{op}}
        \\
        &\qquad+\|D^2_{(x,\theta)}b(t,x,\theta,\omega)\|_{\mathrm{op}}\le M_R,
        \end{aligned}
    \]
    and
    \[
        \begin{aligned}
        &|f(t,x,\theta,\omega)|+|\nabla_xf(t,x,\theta,\omega)|
        \\
        &\qquad+|\nabla_\theta f(t,x,\theta,\omega)|
        +\|D^2_{(x,\theta)}f(t,x,\theta,\omega)\|_{\mathrm{op}}\le M_R.
        \end{aligned}
    \]
    \item For each $\omega$, $g(\cdot,\omega)$ is twice continuously
    differentiable. For every $R>0$ there is $M_R>0$, independent of
    $\omega$, such that
    \[
        |g(x,\omega)|+|\nabla_xg(x,\omega)|
        +\|\nabla_x^2g(x,\omega)\|_{\mathrm{op}}\le M_R,\qquad |x|\le R.
    \]
    For each $x$, the functions $g(x,\cdot)$, $\nabla_xg(x,\cdot)$,
    and $\nabla_x^2g(x,\cdot)$ are measurable.
    \item $h$ is convex and admits a twice continuously differentiable
    extension to an open neighborhood of $U$.
\end{enumerate}
\end{assumption}

The SSM problem $(\mathbf P_S^\tau)$ is the special case of
$(\mathbf P^\tau)$ obtained by setting
\begin{align}
\label{eq:ssm-regularizer-identification}
    b(t,x,\theta,\omega)&=A(t,\omega;\theta)x+B(t,\omega;\theta),\\
    f(t,x,\theta,\omega)&=|Cx-\mathcal F[\omega](t)|^2, \qquad g(x,\omega)=0,\qquad h(\theta)=\frac12|\theta|^2.\nonumber
\end{align}
Assumption~\ref{asp:main_asp} verifies
Assumption~\ref{ass:general-first-variation} for this choice: continuity
and compactness give the measurability and boundedness conditions, the
affine state dynamics have linear growth, and the parameter derivative
bounds give the required locally uniform derivatives. The same parameter
$\tau$ is used in both objectives.

\section{Adjoint Sensitivity and First-Order Training}
\label{sec:adjoint-mirror}

We study first-order optimization of the reduced training objective defined in Section~\ref{sec:form}. Such methods require its variation with respect to the shared parameter path. The adjoint sensitivity gives an efficient backward representation of this variation, and for SSMs it is the continuous-time analogue of backpropagation. We first establish uniform well-posedness of the state-adjoint system and the first-variation formula of the training objective, then express the functional projected-gradient update for SSMs using this representation. Finally, we embed that update in the general Bregman mirror-descent procedure summarized in Algorithm~\ref{alg:ensemble-mirror-descent}.

\subsection{State-adjoint system and first variation}
\label{sec:first-variation}

For an admissible control $u\in\mathcal U[0,T]$, let $x^u$ solve
\eqref{eq:ctrl_sys}. The backward sensitivity equation associated with
the reduced objective is
\begin{equation}
\label{eq:general-adjoint}
\left\{\begin{aligned}
\dot p^u(t,\omega)
&=-D_x b(t,x^u(t,\omega),u(t),\omega)^\top p^u(t,\omega)
  +\nabla_xf(t,x^u(t,\omega),u(t),\omega),\\
p^u(T,\omega)&=-\nabla_xg(x^u(T,\omega),\omega).
\end{aligned}\right.
\end{equation}
This adjoint propagates the sensitivity of the running and terminal costs backward through the state dynamics.

\begin{lemma}[Uniform well-posedness, bounds, and measurability]
\label{lem:uniform-state-adjoint-bounds}
Suppose Assumption~\ref{ass:general-first-variation} holds. For every
$u\in\mathcal U[0,T]$, the state and adjoint equations have unique
solutions, absolutely continuous in $t$ for each $\omega$ and jointly
measurable in $(t,\omega)$. Uniformly in $u,t,\omega$,
\begin{equation}
\label{eq:uniform-state-bound-general}
|x^u(t,\omega)|\le M_X,
\qquad
M_X:=(\|x_0\|_\infty+M_bT)e^{M_bT}.
\end{equation}
On $[0,T]\times\{|x|\le M_X\}\times U\times\Omega$, define
\begin{equation}
\label{eq:general-bound-constants}
B_x:=\sup\|D_x b\|_{\mathrm{op}},\qquad
F_x:=\sup|\nabla_xf|,\qquad
G_x:=\sup_{|x|\le M_X,\,\omega\in\Omega}
|\nabla_xg(x,\omega)|.
\end{equation}
Then
\begin{equation}
\label{eq:uniform-adjoint-bound-general}
|p^u(t,\omega)|\le M_P,
\qquad
M_P:=(G_x+TF_x)e^{B_xT}.
\end{equation}

Moreover, the coefficient maps appearing along the state trajectory,
\begin{equation}
\label{eq:state-adjoint-coefficient-measurability}
(t,\omega)\longmapsto
D_x b(t,x^u(t,\omega),u(t),\omega),
\qquad
(t,\omega)\longmapsto
\nabla_x f(t,x^u(t,\omega),u(t),\omega),
\end{equation}
are jointly measurable, and
\begin{equation}
\label{eq:terminal-gradient-measurability}
\omega\longmapsto
\nabla_x g(x^u(T,\omega),\omega)
\end{equation}
is measurable.

If the growth estimate is available in the sharper form
$|b(t,x,\theta,\omega)|\le c_0+c_1|x|$, with $c_0,c_1\ge0$,
one may instead take
$M_X=(\|x_0\|_\infty+c_0T)e^{c_1T}$ and evaluate the other
constants on that state region. Any displayed derivative supremum
may also be replaced by a known upper bound.
\end{lemma}

\begin{proof}
See Appendix~\ref{app:regularity-state-adjoint}.
\end{proof}

\begin{lemma}[Parameter-derivative measurability]
\label{lem:parameter-derivative-measurability}
Suppose Assumption~\ref{ass:general-first-variation} holds.
Then
\[
(t,\omega,x,\theta)\longmapsto
D_\theta b(t,x,\theta,\omega),
\qquad
(t,\omega,x,\theta)\longmapsto
\nabla_\theta f(t,x,\theta,\omega)
\]
are jointly measurable on $[0,T]\times\Omega\times\mathbb R^n\times U$.

Consequently, if
$\xi:[0,T]\times\Omega\to\mathbb R^n$
is jointly measurable and
$u\in\mathcal U[0,T]$, then
\[
(t,\omega)\longmapsto
D_\theta b(t,\xi(t,\omega),u(t),\omega),
\qquad
(t,\omega)\longmapsto
\nabla_\theta f(t,\xi(t,\omega),u(t),\omega)
\]
are jointly measurable.
\end{lemma}

\begin{proof}
See Appendix~\ref{app:regularity-parameter-derivatives}.
\end{proof}

\begin{lemma}[Parameter gradient of the ensemble Hamiltonian]
\label{lem:Hamiltonian-parameter-derivative}
Suppose Assumption~\ref{ass:general-first-variation} holds.
For every $t\in[0,T]$ and bounded measurable
$\bar\xi,\bar\pi:\Omega\to\mathbb R^n$,
the map
\[
\theta\longmapsto H_t^\tau(\bar\xi,\bar\pi,\theta)
\]
is continuously differentiable on $U$, with
\begin{equation}
\label{eq:general-Hamiltonian-u-derivative}
\nabla_\theta H_t^\tau(\bar\xi,\bar\pi,\theta)
=
\mathbb E_\omega\!\left[
D_\theta b(t,\bar\xi(\omega),\theta,\omega)^\top
\bar\pi(\omega)
-
\nabla_\theta f(t,\bar\xi(\omega),\theta,\omega)
\right]
-
\tau\nabla h(\theta).
\end{equation}

Moreover, if $\xi,\pi:[0,T]\times\Omega\to\mathbb R^n$ are jointly
measurable and uniformly bounded and $u\in\mathcal U[0,T]$, then
\[
t\longmapsto
\nabla_\theta H_t^\tau
\bigl(\xi(t,\cdot),\pi(t,\cdot),u(t)\bigr)
\]
is measurable.
\end{lemma}

\begin{proof}
By Lemma~\ref{lem:parameter-derivative-measurability} and the boundedness
of $\bar\xi,\bar\pi$, the parameter derivatives appearing in the
expectation are measurable and, by
Assumption~\ref{ass:general-first-variation}(4), uniformly dominated for
$\theta\in U$. Hence differentiation under the expectation yields
\eqref{eq:general-Hamiltonian-u-derivative} on $\operatorname{int}U$.
Continuity of the parameter derivatives and dominated convergence then
extend the identity to every $\theta\in U$.

For the second assertion, Lemma~\ref{lem:parameter-derivative-measurability}
and composition imply that
\[
(t,\omega)\longmapsto
D_\theta b(t,\xi(t,\omega),u(t),\omega)^\top\pi(t,\omega)
-\nabla_\theta f(t,\xi(t,\omega),u(t),\omega)
\]
is jointly measurable. Since $\xi$ and $\pi$ are bounded,
Assumption~\ref{ass:general-first-variation}(4) gives a uniform
integrable bound in $\omega$; hence its expectation is measurable in
$t$. Moreover, $t\mapsto\nabla h(u(t))$ is measurable. The claim now
follows from \eqref{eq:general-Hamiltonian-u-derivative}.
\end{proof}

\begin{lemma}[First variation of the ensemble objective]
\label{lem:first-variation-general}
Suppose Assumption~\ref{ass:general-first-variation} holds. For all
$u,v\in\mathcal U[0,T]$, the one-sided feasible directional derivative
\[
dJ^\tau[u](v-u):=\lim_{\varepsilon\downarrow0}
\frac{J^\tau[u+\varepsilon(v-u)]-J^\tau[u]}{\varepsilon}
\]
exists and satisfies
\begin{equation}
\label{eq:general-first-variation}
dJ^\tau[u](v-u)
=-\int_0^T\left\langle
\nabla_\theta H_t^\tau(x^u(t,\cdot),p^u(t,\cdot),u(t)),
v(t)-u(t)\right\rangle\,\dd t.
\end{equation}
\end{lemma}

\begin{proof}
See Appendix~\ref{app:proof_first_variation_general}.
\end{proof}

\subsection{SSM specialization and backpropagation}
\label{subsec:training_gradient}

For the normalized SSM problem $(\mathbf P_S^\tau)$,
the identification \eqref{eq:ssm-regularizer-identification}
specializes \eqref{eq:general-adjoint} to
\begin{equation}
\label{eq:adj_first_variation}
\left\{\begin{aligned}
\dot p^u(t,\omega)&=-A(t,\omega;u(t))^\top p^u(t,\omega)
 +2C^\top\bigl(Cx^u(t,\omega)-\mathcal F[\omega](t)\bigr),\\
p^u(T,\omega)&=0.
\end{aligned}\right.
\end{equation}
The corresponding Hamiltonian and its parameter derivative are
\begin{align}
\label{eq:Ht_first_variation}
H_t^\tau(\xi,\pi,\theta)
&=\mathbb E_\omega\!\left[
\pi(\omega)\cdot\bigl(A(t,\omega;\theta)\xi(\omega)+B(t,\omega;\theta)\bigr)
-|C\xi(\omega)-\mathcal F[\omega](t)|^2\right]
 -\frac\tau2|\theta|^2,\\
\label{eq:DuHt_first_variation}
\begin{split}
\left\langle\nabla_\theta H_t^\tau(\xi,\pi,\theta),\zeta\right\rangle
&{}=\mathbb E_\omega\!\left[
\pi(\omega)\cdot\bigl(
D_\theta A(t,\omega;\theta)[\zeta]\,\xi(\omega)
+D_\theta B(t,\omega;\theta)[\zeta]\bigr)\right]-\tau\,\theta\cdot\zeta,
\qquad\zeta\in\mathbb R^m.
\end{split}
\end{align}

\begin{corollary}[First variation of the SSM training objective]
\label{lmm:first_variation}
Suppose Assumption~\ref{asp:main_asp} holds. For admissible $u,v$,
\begin{equation}
\label{eq:first_variation_J}
dJ^\tau[u](v-u)=-\int_0^T\left\langle
\nabla_\theta H_t^\tau(x^u(t,\cdot),p^u(t,\cdot),u(t)),
v(t)-u(t)\right\rangle\,\dd t,
\end{equation}
where $p^u$ solves \eqref{eq:adj_first_variation} and the Hamiltonian
gradient is given by \eqref{eq:DuHt_first_variation}.
\end{corollary}

\begin{proof}
The continuity and compactness conditions in
Assumption~\ref{asp:main_asp} give measurable, uniformly bounded
$x_0$ and target trajectories. The SSM drift satisfies
$|b|\le M_A|x|+M_B$, and its derivatives through order two are
bounded on each bounded state region by the structural constants
in Assumption~\ref{asp:main_asp}. The quadratic running cost has
$\nabla_x f=2C^\top(Cx-\mathcal F[\omega](t))$,
$\nabla_x^2 f=2C^\top C$, and zero parameter derivatives.
Together with $g=0$ and $h(\theta)=|\theta|^2/2$, these facts verify
Assumption~\ref{ass:general-first-variation}. Substitution into
Lemma~\ref{lem:first-variation-general} proves the result. 
\end{proof}

The forward state equation is continuous-time model evaluation, or
the forward pass. The backward adjoint equation propagates reverse
sensitivities, giving the continuous-time counterpart of
backpropagation. In particular, the functional-gradient density of the reduced objective is
\begin{equation}
\label{eq:training-gradient-density}
G[u](t):=-\nabla_\theta H_t^\tau(x^u(t,\cdot),p^u(t,\cdot),u(t)),
\qquad
dJ^\tau[u](v-u)=\int_0^T\langle G[u](t),v(t)-u(t)\rangle\,\dd t.
\end{equation}
The Hamiltonian gradient includes the ensemble expectation.
Thus the state-adjoint computation evaluates the gradient entering
the functional projected-gradient update
\begin{equation}
\label{eq:ssm-projected-gradient-update}
\begin{aligned}
u^{(k+1)}(t)
&=P_U\bigl(u^{(k)}(t)-\eta G[u^{(k)}](t)\bigr)\\
&=P_U\!\left(u^{(k)}(t)+\eta\nabla_\theta H_t^\tau
\bigl(x^{u^{(k)}}(t,\cdot),p^{u^{(k)}}(t,\cdot),u^{(k)}(t)\bigr)\right),
\end{aligned}
\end{equation}
where $P_U$ is Euclidean projection onto $U$ and $\eta>0$ is the
stepsize.

\begin{remark}[Relation to tied-parameter training]
\label{rem:tied-parameter-training}
The parameter-path problem $(\mathbf P_S^\tau)$ optimizes over
$u:[0,T]\to U$. The SSM architectures displayed in
Section~\ref{subsec:ssm_examples}, however, use finite-dimensional
parameters that are shared across sequence positions.

Consider first the tied-parameter restriction
\[
u_\theta(t)\equiv\theta,
\qquad
\widehat J^\tau(\theta):=J^\tau[u_\theta].
\]
For $\theta\in\operatorname{int}U$ and every direction
$\zeta\in\mathbb R^m$, Corollary~\ref{lmm:first_variation} gives
\[
D\widehat J^\tau(\theta)[\zeta]
=
\int_0^T
\langle G[u_\theta](t),\zeta\rangle\,\dd t,
\]
and hence
\[
\nabla_\theta\widehat J^\tau(\theta)
=
\int_0^T G[u_\theta](t)\,\dd t.
\]
At boundary points of $U$, the first identity holds for feasible one-sided
directions. Thus $G[u_\theta](t)$ represents the contribution from the
sequence coordinate $t$, while the gradient with respect to a parameter
shared across all sequence positions is obtained by accumulating these
contributions over the interval.

The same relation extends naturally to more general finite-dimensional
parameterizations. Indeed, let $\Theta\subset\mathbb R^q$ be an open
parameter domain and let
\[
\theta\mapsto u_\theta(\cdot)\in\mathcal U[0,T]
\]
be differentiable as a map from $\Theta$ into
$L^2(0,T;\mathbb R^m)$. Defining again
\[
\widehat J^\tau(\theta):=J^\tau[u_\theta],
\]
one has, for every $\zeta\in\mathbb R^q$,
\[
D\widehat J^\tau(\theta)[\zeta]
=
\int_0^T
\left\langle
G[u_\theta](t),
D_\theta u_\theta(t)[\zeta]
\right\rangle\,\dd t.
\]
For the tied restriction $u_\theta(t)\equiv\theta$, one has $q=m$ and
$D_\theta u_\theta(t)[\zeta]=\zeta$, so the general formula reduces to
the accumulation formula above.

This is the continuous-time counterpart of accumulating reverse-mode
sensitivities for parameters shared across sequence positions. If $\mu$ is
chosen as the empirical distribution of the training sequences, the ensemble
expectation is the full-batch average, while minibatch training replaces it
by a stochastic estimate. Thus the functional projected-gradient update acts
on the parameter-path objective, while tied or more general shared
finite-dimensional parameterizations inherit their gradients through the
chain rule above.
\end{remark}

\subsection{Bregman mirror-descent formulation}
\label{subsec:bregman-mirror-descent}

We now embed the functional projected-gradient update in a general first-order
method for $(\mathbf P^\tau)$, using the geometry of its regularizer.

\begin{definition}[Bregman divergence]
\label{def:bregman-divergence}
For convex differentiable $h$, the pointwise Bregman divergence is
\begin{equation}
\label{eq:bregman-divergence}
D_h(\vartheta\mid\theta):=
h(\vartheta)-h(\theta)-\nabla h(\theta)\cdot(\vartheta-\theta),
\qquad\theta,\vartheta\in U.
\end{equation}
For controls $u,v\in\mathcal U[0,T]$, its integrated version is
\begin{equation}
\label{eq:integrated-bregman-divergence}
\mathcal D_h[v\mid u]:=\int_0^T D_h(v(t)\mid u(t))\,\dd t.
\end{equation}
\end{definition}

\begin{assumption}[Uniform convexity of the mirror map]
\label{ass:mirror-map}
There exists $\sigma_h>0$ such that
\begin{equation}
\label{eq:uniform-convexity-h}
D_h(\vartheta\mid\theta)\ge\frac{\sigma_h}{2}|\vartheta-\theta|^2,
\qquad\theta,\vartheta\in U.
\end{equation}
Consequently, for admissible $u,v$,
\begin{equation}
\label{eq:integrated-uniform-convexity}
\mathcal D_h[v\mid u]\ge\frac{\sigma_h}{2}\|v-u\|_{L^2}^2.
\end{equation}
\end{assumption}

By Lemma~\ref{lem:first-variation-general}, the general objective has
the same gradient representation:
\begin{equation}
\label{eq:first-variation-Jtau}
dJ^\tau[u](v-u)
=
-\int_0^T
\left\langle
\nabla_\theta H_t^\tau
\bigl(x^u(t,\cdot),p^u(t,\cdot),u(t)\bigr),
v(t)-u(t)
\right\rangle
\,\dd t.
\end{equation}
Minimizing its first-order approximation with a Bregman penalty is
equivalent to the pointwise mirror step
\begin{equation}
\label{eq:mirror-update-general}
\begin{aligned}
u^{(k+1)}(t)\in\argmax_{\theta\in U}\Bigl\{&
\left\langle
\nabla_\theta H_t^\tau
\bigl(
x^{u^{(k)}}(t,\cdot),
p^{u^{(k)}}(t,\cdot),
u^{(k)}(t)
\bigr),
\theta-u^{(k)}(t)
\right\rangle
\\
&\qquad
-\lambda D_h\bigl(\theta\mid u^{(k)}(t)\bigr)
\Bigr\},
\end{aligned}
\end{equation}
where $\lambda>0$ is the inverse stepsize.
By Lemma~\ref{lem:Hamiltonian-parameter-derivative}, the
Hamiltonian-gradient process in \eqref{eq:mirror-update-general} is
measurable. The following lemma shows that the pointwise mirror
problem admits a unique measurable maximizer and hence defines an
admissible control.

\begin{lemma}[Well-definedness of the mirror step]
\label{lem:mirror-step-measurability}
Suppose Assumptions~\ref{ass:general-first-variation}
and~\ref{ass:mirror-map} hold. Let
\[
\xi,\pi:[0,T]\times\Omega\to\mathbb R^n
\]
be jointly measurable and uniformly bounded, and let
$u\in\mathcal U[0,T]$. Define
\begin{equation}
\label{eq:measurable-Hamiltonian-gradient}
Q_{\xi,\pi,u}(t)
:=
\nabla_\theta H_t^\tau
\bigl(\xi(t,\cdot),\pi(t,\cdot),u(t)\bigr).
\end{equation}
For every
$\lambda>0$ the pointwise maximization problem
\begin{equation}
\label{eq:generic-mirror-maximization}
\argmax_{\theta\in U}
\left\{
\left\langle
Q_{\xi,\pi,u}(t),\theta-u(t)
\right\rangle
-\lambda D_h(\theta\mid u(t))
\right\}
\end{equation}
has a unique maximizer for every $t$, and this maximizer is measurable
in $t$. In particular, the update
\eqref{eq:mirror-update-general} defines
$u^{(k+1)}\in\mathcal U[0,T]$.
\end{lemma}

\begin{proof}
By Lemma~\ref{lem:Hamiltonian-parameter-derivative},
$t\mapsto Q_{\xi,\pi,u}(t)$ is measurable. Define
\[
\Phi(t,\theta)
:=
\left\langle
Q_{\xi,\pi,u}(t),\theta-u(t)
\right\rangle
-
\lambda D_h(\theta\mid u(t)),
\qquad
(t,\theta)\in[0,T]\times U.
\]
For each fixed $\theta\in U$, the map
$t\mapsto\Phi(t,\theta)$ is measurable, since
$Q_{\xi,\pi,u}$ and $u$ are measurable and $h,\nabla h$ are
continuous. For each fixed $t\in[0,T]$, the map
$\theta\mapsto\Phi(t,\theta)$ is continuous on $U$.

Moreover, by Assumption~\ref{ass:mirror-map},
\[
\theta\longmapsto D_h(\theta\mid u(t))
\]
is $\sigma_h$-strongly convex. Hence, for each fixed $t$,
$\Phi(t,\cdot)$ is $\lambda\sigma_h$-strongly concave on the convex
set $U$. Since $U$ is compact, $\Phi(t,\cdot)$ attains its maximum,
and strong concavity makes this maximizer unique. Thus the
correspondence
\[
\Gamma(t)
:=
\argmax_{\theta\in U}\Phi(t,\theta)
\]
is nonempty and singleton-valued for every $t$.

Since $\Phi$ is a Carath\'eodory function (measurable in $t$ for each fixed $\theta$ and continuous in $\theta$ for each fixed $t$) and $U$ is a compact metric space, the measurable maximum theorem gives a measurable selection
\[
\theta^*(t)\in\Gamma(t).
\]
Because $\Gamma(t)$ is a singleton, this selection is necessarily
the unique maximizer of $\Phi(t,\cdot)$. Therefore
$t\mapsto\theta^*(t)$ is measurable.

Finally, for the mirror-descent iterate take
\[
\xi(t,\omega)=x^{u^{(k)}}(t,\omega),
\qquad
\pi(t,\omega)=p^{u^{(k)}}(t,\omega),
\qquad
u=u^{(k)}.
\]
Lemma~\ref{lem:uniform-state-adjoint-bounds} gives joint measurability and uniform boundedness of the state and adjoint. Therefore \eqref{eq:mirror-update-general} defines a unique measurable $U$-valued function $u^{(k+1)}$, and hence
\[
u^{(k+1)}\in\mathcal U[0,T].
\]
\end{proof}

The resulting functional mirror-descent procedure for the ensemble objective is
summarized in
Algorithm~\ref{alg:ensemble-mirror-descent}.

\medskip

\begin{algorithm}[H]
\DontPrintSemicolon
\caption{Bregman mirror descent for ensemble control}
\label{alg:ensemble-mirror-descent}
\KwIn{Inverse stepsize $\lambda>0$ and initial control
$u^{(0)}\in\mathcal U[0,T]$.}
\textbf{For} $k=0,1,2,\ldots$:\;
\Indp
\textbf{Forward step:} solve \eqref{eq:ctrl_sys} for
$x^{u^{(k)}}$ with control $u^{(k)}$.\;
\textbf{Adjoint step:} solve \eqref{eq:general-adjoint}
backward for $p^{u^{(k)}}$.\;
\textbf{Mirror step:} obtain $u^{(k+1)}(t)$ from
\eqref{eq:mirror-update-general} for a.e.\ $t\in[0,T]$.\;
\Indm
\end{algorithm}

\begin{remark}[Functional projected gradient descent as the Euclidean case]
\label{rem:euclidean-mirror-ssm}
For $h(\theta)=|\theta|^2/2$,
$D_h(\vartheta\mid\theta)=|\vartheta-\theta|^2/2$.
Completing the square in \eqref{eq:mirror-update-general} gives
exactly the SSM projected-gradient update
\eqref{eq:ssm-projected-gradient-update} with $\eta=1/\lambda$.
\end{remark}

The next section combines state, adjoint, and Hamiltonian-gradient
stability to bound both sides of the reduced objective's curvature.
This yields relative smoothness, energy dissipation, and relative
convexity under sufficient regularization.
Section~\ref{sec:ssm-consequences} specializes these estimates to SSMs.

\section{Convergence Analysis for Ensemble Mirror Descent}
\label{sec:ensemble-mirror-convergence}

We quantify the convergence of Algorithm~\ref{alg:ensemble-mirror-descent}.
State-adjoint stability first gives Lipschitz continuity of the Hamiltonian gradient and a two-sided curvature estimate for the reduced objective. Its upper bound implies energy dissipation, while its lower bound gives relative convexity under sufficient regularization. In this relatively convex regime, we also establish existence of optimal controls, with uniqueness under relative strong convexity. The three-point inequality then yields convergence rates from this stability-generated curvature bound.
Throughout this section, Assumption~\ref{ass:general-first-variation} holds, and \(M_X,M_P\) are the bounds from Lemma~\ref{lem:uniform-state-adjoint-bounds}.

\subsection{State and adjoint stability}
\label{subsec:ensemble-stability}

Recall the state-adjoint system associated with an admissible control:
\begin{equation}
\label{eq:ensemble-state-adjoint-u}
\left\{\begin{aligned}
\dot x^u(t,\omega)&=b(t,x^u(t,\omega),u(t),\omega),\\
x^u(0,\omega)&=x_0(\omega),\\
\dot p^u(t,\omega)&=-D_x b(t,x^u(t,\omega),u(t),\omega)^\top p^u(t,\omega)
+\nabla_x f(t,x^u(t,\omega),u(t),\omega),\\
p^u(T,\omega)&=-\nabla_xg(x^u(T,\omega),\omega).
\end{aligned}\right.
\end{equation}
All derivative suprema below are over
$[0,T]\times\{|x|\le M_X\}\times U\times\Omega$, except that
terminal-cost suprema are over $\{|x|\le M_X\}\times\Omega$.
We retain $B_x$ and $G_x$ from \eqref{eq:general-bound-constants} and set
\begin{equation}
\label{eq:ensemble-convergence-bounds}
B_\theta:=\sup\|D_\theta b\|_{\mathrm{op}},
\qquad
L_{ij}^{b}:=\sup\|D_iD_jb\|_{\mathrm{op}},
\qquad
L_{ij}^{f}:=\sup\|D_iD_jf\|_{\mathrm{op}},
\quad i,j\in\{x,\theta\}.
\end{equation}
These finite second-derivative bounds control the changes of the
first derivatives on the convex state-parameter region. In particular,
\begin{align}
\label{eq:ensemble-convergence-lipschitz}
\|D_x b(t,x,\theta,\omega)-D_x b(t,\widetilde x,\vartheta,\omega)\|_{\mathrm{op}}
&\le L_{xx}^b|x-\widetilde x|+L_{x\theta}^b|\theta-\vartheta|,\\
|\nabla_xf(t,x,\theta,\omega)-\nabla_xf(t,\widetilde x,\vartheta,\omega)|
&\le L_{xx}^f|x-\widetilde x|+L_{x\theta}^f|\theta-\vartheta|.\nonumber
\end{align}
Likewise, with $L_g:=\sup\|\nabla_x^2g\|_{\mathrm{op}}$,
\begin{equation}
\label{eq:terminal-gradient-lipschitz}
|\nabla_xg(x,\omega)|\le G_x,\qquad
|\nabla_xg(x,\omega)-\nabla_xg(\widetilde x,\omega)|\le L_g|x-\widetilde x|.
\end{equation}
For later use define
\begin{align}
\label{eq:general-stability-constants}
K_X&:=B_\theta e^{B_xT},& C_x&:=TK_X,\nonumber\\
A_x&:=M_PL_{xx}^b+L_{xx}^f,&
A_\theta&:=M_PL_{x\theta}^b+L_{x\theta}^f,\\
C_p&:=Te^{B_xT}\bigl(L_gK_X+A_xTK_X+A_\theta\bigr).&&\nonumber
\end{align}

\begin{lemma}[Stability of the state and adjoint]
\label{lem:ensemble-state-adjoint-stability}
Suppose Assumption~\ref{ass:general-first-variation} holds. For all
$u,v\in\mathcal U[0,T]$, the explicit constants in
\eqref{eq:general-stability-constants} give
\begin{align}
\label{eq:ensemble-state-stability}
\|x^v-x^u\|_{L^2}&\le C_x\|v-u\|_{L^2},\\
\label{eq:ensemble-adjoint-stability}
\|p^v-p^u\|_{L^2}&\le C_p\|v-u\|_{L^2}.
\end{align}
\end{lemma}

\begin{proof}
Set $\delta u=v-u$, $\delta x=x^v-x^u$, and $\delta p=p^v-p^u$.
Subtracting the state integral equations and using the mean-value
theorem gives
\[
|\delta x(t,\omega)|\le\int_0^t
\bigl(B_x|\delta x(s,\omega)|+B_\theta|\delta u(s)|\bigr)\,\dd s.
\]
Gronwall's inequality yields the pointwise estimate
\begin{equation}
\label{eq:general-pointwise-state-difference}
|\delta x(t,\omega)|\le K_X\int_0^t|\delta u(s)|\,\dd s
\le K_X\sqrt T\,\|\delta u\|_{L^2}.
\end{equation}
Taking the ensemble and time $L^2$ norm proves
\eqref{eq:ensemble-state-stability} with $C_x=TK_X$.

For the adjoint difference, subtracting the two backward equations gives
\begin{align*}
\dot{\delta p}(t,\omega)={}&
-D_x b(t,x^v(t,\omega),v(t),\omega)^\top\delta p(t,\omega)\\
&-\bigl(D_x b(t,x^v(t,\omega),v(t),\omega)
-D_x b(t,x^u(t,\omega),u(t),\omega)\bigr)^\top p^u(t,\omega)\\
&+\nabla_x f(t,x^v(t,\omega),v(t),\omega)
-\nabla_x f(t,x^u(t,\omega),u(t),\omega),
\end{align*}
with $|\delta p(T,\omega)|\le L_g|\delta x(T,\omega)|$.
By \eqref{eq:ensemble-convergence-lipschitz} and the adjoint bound in
Lemma~\ref{lem:uniform-state-adjoint-bounds}, the forcing difference
is bounded by $A_x|\delta x|+A_\theta|\delta u|$.
Backward Gronwall and \eqref{eq:general-pointwise-state-difference}
therefore give
\begin{align*}
|\delta p(t,\omega)|
&\le e^{B_xT}\left(L_g|\delta x(T,\omega)|
+\int_t^T\bigl(A_x|\delta x(s,\omega)|+A_\theta|\delta u(s)|\bigr)\,\dd s\right)\\
&\le e^{B_xT}(L_gK_X+A_xTK_X+A_\theta)
\sqrt T\,\|\delta u\|_{L^2}.
\end{align*}
Taking the ensemble and time $L^2$ norm proves
\eqref{eq:ensemble-adjoint-stability} with the stated $C_p$.
\end{proof}

\subsection{Hamiltonian-gradient stability and relative curvature}
\label{subsec:ensemble-relative-smoothness}

Define the bounded state and adjoint sets
\[
\begin{aligned}
\mathcal X&:=\{\xi\in L^2(\Omega;\mathbb R^n):|\xi(\omega)|\le M_X
\text{ for a.e. }\omega\},\\
\mathcal P&:=\{\pi\in L^2(\Omega;\mathbb R^n):|\pi(\omega)|\le M_P
\text{ for a.e. }\omega\}.
\end{aligned}
\]
They are convex and contain the state and adjoint slices for every
admissible control. In the notation of
\eqref{eq:ensemble-convergence-bounds}, set
\begin{equation}
\label{eq:CH-general}
C_H:=\max\left\{
M_PL_{\theta x}^b+L_{\theta x}^f,\ B_\theta,\
M_PL_{\theta\theta}^b+L_{\theta\theta}^f\right\}.
\end{equation}

\begin{lemma}[Lipschitz continuity of the Hamiltonian gradient]
\label{lem:ensemble-Hamiltonian-gradient-Lipschitz}
Suppose Assumption~\ref{ass:general-first-variation} holds. For
$\xi,\widetilde\xi\in\mathcal X$, $\pi,\widetilde\pi\in\mathcal P$,
$\theta,\widetilde\theta\in U$, and $t\in[0,T]$,
\begin{align}
\label{eq:ensemble-Hamiltonian-gradient-Lipschitz}
&|\nabla_\theta H_t^0(\xi,\pi,\theta)
-\nabla_\theta H_t^0(\widetilde\xi,\widetilde\pi,\widetilde\theta)|\nonumber\\
&\qquad\le C_H\left(
\|\xi-\widetilde\xi\|_{L^2(\Omega)}
+\|\pi-\widetilde\pi\|_{L^2(\Omega)}+|\theta-\widetilde\theta|\right).
\end{align}
\end{lemma}

\begin{proof}
By \eqref{eq:general-Hamiltonian-u-derivative},
\[
\nabla_\theta H_t^0(\xi,\pi,\theta)=\mathbb E_\omega
\left[D_\theta b(t,\xi(\omega),\theta,\omega)^\top\pi(\omega)
-\nabla_\theta f(t,\xi(\omega),\theta,\omega)\right].
\]
Change the state, adjoint, and parameter in that order. The three
differences are respectively bounded by
\[
\begin{aligned}
(M_PL_{\theta x}^b+L_{\theta x}^f)
\mathbb E_\omega|\xi-\widetilde\xi|, \qquad B_\theta\mathbb E_\omega|\pi-\widetilde\pi|,\qquad (M_PL_{\theta\theta}^b+L_{\theta\theta}^f)
|\theta-\widetilde\theta|.
\end{aligned}
\]
The mean-value theorem justifies the first and third estimates;
linearity in the adjoint gives the second. Since $\mu$ is a probability
measure, Cauchy--Schwarz bounds each expectation by its $L^2(\Omega)$
norm. Summing the three estimates proves the result.
\end{proof}

Define the stability-generated defect
\begin{equation}
\label{eq:kappa-stab-general}
\kappa_{\mathrm{stab}}:=\frac{C_H(C_x+C_p+1)}{\sigma_h}.
\end{equation}
The absolute-value estimate for the unregularized gradient variation
controls both sides of the reduced objective's curvature.

\begin{theorem}[Two-sided relative curvature and smoothness]
\label{thm:ensemble-relative-smoothness}
Suppose Assumptions~\ref{ass:general-first-variation} and
\ref{ass:mirror-map} hold. For all $u,v\in\mathcal U[0,T]$,
\begin{equation}
\label{eq:ensemble-relative-smoothness}
\begin{aligned}
dJ^\tau[u](v-u)+(\tau-\kappa_{\mathrm{stab}})\mathcal D_h[v\mid u]
&\le J^\tau[v]-J^\tau[u]\\
&\le dJ^\tau[u](v-u)+(\tau+\kappa_{\mathrm{stab}})\mathcal D_h[v\mid u].
\end{aligned}
\end{equation}
The explicit relative-smoothness constant is
\begin{equation}
\label{eq:Lrel-general}
L_{\mathrm{rel}}:=\tau+\kappa_{\mathrm{stab}}.
\end{equation}
The upper bound is equivalently
\begin{align}
\label{eq:ensemble-relative-smoothness-H}
J^\tau[v]-J^\tau[u]\le{}&-\int_0^T\left\langle
\nabla_\theta H_t^\tau(x^u(t,\cdot),p^u(t,\cdot),u(t)),
v(t)-u(t)\right\rangle\,\dd t +L_{\mathrm{rel}}\mathcal D_h[v\mid u].
\end{align}
Define
\begin{equation}
\label{eq:mu-stab-general}
\mu_{\mathrm{stab}}:=\tau-\kappa_{\mathrm{stab}}.
\end{equation}
If $\tau=\kappa_{\mathrm{stab}}$, then $J^\tau$ is relatively convex;
if $\tau>\kappa_{\mathrm{stab}}$, it is relatively strongly convex
with modulus $\mu_{\mathrm{stab}}$.
\end{theorem}

\begin{proof}
Fix $u,v$ and set $\delta u=v-u$, $u^s=u+s\delta u$ for
$0\le s\le1$.
We first justify integrating the first variation along the
segment $s\mapsto u^s$. Set $\varphi(s):=J^\tau[u^s]$ and
\[
\psi(s):=-\int_0^T
\langle \nabla_\theta H_t^\tau(x^{u^s}(t,\cdot),p^{u^s}(t,\cdot),u^s(t)),\delta u(t)\rangle\,\dd t,
\qquad 0\le s\le1.
\]
For $\varepsilon>0$ we have
$u^s+\varepsilon\delta u=u^s+\frac{\varepsilon}{1-s}(v-u^s)$ when
$s\in[0,1)$, and
$u^s-\varepsilon\delta u=u^s+\frac{\varepsilon}{s}(u-u^s)$ when
$s\in(0,1]$. Applying Lemma~\ref{lem:first-variation-general} at the
base point $u^s$ in the feasible directions $v-u^s=(1-s)\delta u$ and
$u-u^s=-s\delta u$ therefore shows that $\varphi$ has right derivative
$\psi(s)$ at every $s\in[0,1)$ and left derivative $\psi(s)$ at every
$s\in(0,1]$. Hence $\varphi$ is continuous on $[0,1]$ and
differentiable on $(0,1)$ with $\varphi'=\psi$. Moreover, with
$F_\theta:=\sup|\nabla_\theta f|$ over the region used in
\eqref{eq:ensemble-convergence-bounds},
formula~\eqref{eq:general-Hamiltonian-u-derivative} gives
\[
|\nabla_\theta H_t^\tau(\xi,\pi,\theta)|
\le M_Q:=B_\theta M_P+F_\theta+\tau\max_{\vartheta\in U}|\nabla h(\vartheta)|,
\qquad \xi\in\mathcal X,\ \pi\in\mathcal P,\ \theta\in U,
\]
so $|\psi(s)|\le M_Q\sqrt T\,\|\delta u\|_{L^2}$. By the mean-value
theorem, $\varphi$ is Lipschitz on $[0,1]$, hence absolutely
continuous; $\psi$ is measurable as its a.e.\ derivative; and
$\varphi(1)-\varphi(0)=\int_0^1\psi(s)\,\dd s$. In other words,
\begin{equation}
\label{eq:J-segment-integral}
J^\tau[v]-J^\tau[u]=-\int_0^1\int_0^T
\langle \nabla_\theta H_t^\tau(x^{u^s}(t,\cdot),p^{u^s}(t,\cdot),u^s(t)),\delta u(t)\rangle\,\dd t\,\dd s.
\end{equation}
Substituting the first variation at $u$ and rearranging gives
\begin{align*}
J^\tau[v]-J^\tau[u]-dJ^\tau[u](v-u) = &-\int_0^1\int_0^T
\langle \nabla_\theta H_t^\tau(x^{u^s}(t,\cdot),p^{u^s}(t,\cdot),u^s(t))\\
&\qquad-\nabla_\theta H_t^\tau(x^u(t,\cdot),p^u(t,\cdot),u(t)),\delta u(t)\rangle\,\dd t\,\dd s.
\end{align*}
Since $H_t^\tau=H_t^0-\tau h$, the right-hand side splits into $I_2 -I_1$, where
\[
I_1 := \int_0^1\int_0^T
\langle \nabla_\theta H_t^0(x^{u^s}(t,\cdot),p^{u^s}(t,\cdot),u^s(t))-\nabla_\theta H_t^0(x^u(t,\cdot),p^u(t,\cdot),u(t)),\delta u(t)\rangle\,\dd t\,\dd s,
\]
and
\[
I_2 := \tau \int_0^1\int_0^T
\langle \nabla h(u^s(t)) - \nabla h(u(t)), \delta u(t) \rangle \, \dd t \, \dd s.
\]
Both are iterated integrals, with the time integral inside. The
inner integral of $I_2$ is continuous in $s$ by dominated convergence.
The inner integral of $I_1$ equals the inner integral of $I_2$ minus
$\psi(s)-\psi(0)$, so it is measurable in $s$. Hence $I_1$ and $I_2$
are well defined.
By the fundamental theorem of calculus applied to the convex function $h$, we have
\[
I_2 =\tau\mathcal D_h[v\mid u].
\]
Thus the remainder identity is
\begin{equation}
\label{eq:relative-smoothness-remainder}
J^\tau[v]-J^\tau[u]-dJ^\tau[u](v-u)
=-I_1+\tau\mathcal D_h[v\mid u].
\end{equation}
The two stability lemmas and Cauchy--Schwarz give
\begin{align*}
|I_1|&\le C_H\int_0^1
\bigl(\|x^{u^s}-x^u\|_{L^2}+\|p^{u^s}-p^u\|_{L^2}
+\|u^s-u\|_{L^2}\bigr)\|\delta u\|_{L^2}\,\dd s\\
&\le\frac{C_H(C_x+C_p+1)}2\|\delta u\|_{L^2}^2\\
&\le\kappa_{\mathrm{stab}}\mathcal D_h[v\mid u],
\end{align*}
where the last step uses
$\mathcal D_h[v\mid u]\ge(\sigma_h/2)\|\delta u\|_{L^2}^2$.
Hence $-\kappa_{\mathrm{stab}}\mathcal D_h[v\mid u]\le-I_1
\le\kappa_{\mathrm{stab}}\mathcal D_h[v\mid u]$.
Substitution into \eqref{eq:relative-smoothness-remainder} proves both
sides of \eqref{eq:ensemble-relative-smoothness}.
\end{proof}

\begin{proposition}[Existence and uniqueness of optimal controls]
\label{prop:existence-uniqueness}
Suppose Assumptions~\ref{ass:general-first-variation} and~\ref{ass:mirror-map} hold.

\begin{enumerate}
\item If $\tau\ge \kappa_{\mathrm{stab}}$, then $(\mathbf P^\tau)$ admits a
minimizer.
\item If $\tau>\kappa_{\mathrm{stab}}$, then this minimizer is unique
up to equality almost everywhere.
\end{enumerate}
\end{proposition}

\begin{proof}
Since \(U\) is compact, every admissible control belongs to \(L^2(0,T;\mathbb R^m)\). Identifying controls that agree almost everywhere, we regard
$$ \mathcal U[0,T] = \left\{ u\in L^2(0,T;\mathbb R^m): u(t)\in U\ \text{for a.e. }t \right\} $$
as a subset of \(L^2(0,T;\mathbb R^m)\). Since \(U\subset\mathbb R^m\) is nonempty, compact, and convex, \(\mathcal U[0,T]\) is nonempty, bounded, closed, and convex. As \(L^2(0,T;\mathbb R^m)\) is a Hilbert space, \(\mathcal U[0,T]\) is weakly compact.

We first record norm continuity of the reduced objective on $\mathcal U[0,T]$.
By Lemmas~\ref{lem:uniform-state-adjoint-bounds} and~\ref{lem:Hamiltonian-parameter-derivative}, compactness of $U$, and the derivative bounds in
Assumption~\ref{ass:general-first-variation}, the constant
$M_Q<\infty$ from the proof of
Theorem~\ref{thm:ensemble-relative-smoothness} satisfies
\[
\left|
\nabla_\theta H_t^\tau
\bigl(x^u(t,\cdot),p^u(t,\cdot),u(t)\bigr)
\right|
\le M_Q
\]
for every admissible $u$ and for a.e. $t\in[0,T]$. For $u,v\in\mathcal U[0,T]$,
set $u^s=u+s(v-u)$. Integrating the first variation along this feasible
segment, as in \eqref{eq:J-segment-integral}, gives
\[
\begin{aligned}
|J^\tau[v]-J^\tau[u]|
&\le
\int_0^1\int_0^T
M_Q\,|v(t)-u(t)|\,\dd t\,\dd s \le
M_Q\sqrt T\,\|v-u\|_{L^2}.
\end{aligned}
\]
Thus $J^\tau$ is norm continuous on $\mathcal U[0,T]$.

If $\tau\ge\kappa_{\mathrm{stab}}$, Theorem~\ref{thm:ensemble-relative-smoothness} gives, for all
$u,v\in\mathcal U[0,T]$,
\[
J^\tau[v]-J^\tau[u]
\ge
dJ^\tau[u](v-u),
\]
because $\mu_{\mathrm{stab}}=\tau-\kappa_{\mathrm{stab}}\ge0$ and
$\mathcal D_h[v\mid u]\ge0$. This first-order inequality implies that $J^\tau$ is
convex on $\mathcal U[0,T]$. Since $\mathcal U[0,T]$ is norm closed, norm continuity and convexity
imply that every sublevel set of $J^\tau$ in $\mathcal U[0,T]$ is
norm closed and convex in $L^2(0,T;\mathbb R^m)$, hence weakly closed.
Thus $J^\tau$ is weakly lower semicontinuous on $\mathcal U[0,T]$
and attains its minimum on this weakly compact set, proving existence.

Now suppose $\tau>\kappa_{\mathrm{stab}}$ and let $u^\star,v^\star$ be
two minimizers. Since
$s\mapsto J^\tau[u^\star+s(v^\star-u^\star)]$ has a minimum at $s=0$,
\[
dJ^\tau[u^\star](v^\star-u^\star)\ge0.
\]
Applying the lower relative-curvature bound of Theorem~\ref{thm:ensemble-relative-smoothness} yields
\[
0
=
J^\tau[v^\star]-J^\tau[u^\star]
\ge
dJ^\tau[u^\star](v^\star-u^\star)
+
\mu_{\mathrm{stab}}\mathcal D_h[v^\star\mid u^\star]
\ge
\mu_{\mathrm{stab}}\mathcal D_h[v^\star\mid u^\star].
\]
Since $\mu_{\mathrm{stab}}>0$, we have
$\mathcal D_h[v^\star\mid u^\star]=0$. By Assumption~\ref{ass:mirror-map},
\[
0
=
\mathcal D_h[v^\star\mid u^\star]
\ge
\frac{\sigma_h}{2}
\|v^\star-u^\star\|_{L^2}^2,
\]
so $u^\star=v^\star$ almost everywhere.
\end{proof}

\begin{remark}
At the critical value $\tau=\kappa_{\mathrm{stab}}$, the argument above
gives existence but not uniqueness in general. The curvature bound guarantees relative strong convexity,
and hence uniqueness, when $\tau>\kappa_{\mathrm{stab}}$.
\end{remark}

\subsection{Energy dissipation}
\label{subsec:ensemble-energy-dissipation}

Relative smoothness implies monotone decrease of the objective along
the mirror-descent iterates. For concision, in this and the rate
proof below write
\begin{equation}
\label{eq:iterate-Hamiltonian-gradient}
Q^{(k)}(t):=\nabla_\theta H_t^\tau
\bigl(x^{u^{(k)}}(t,\cdot),p^{u^{(k)}}(t,\cdot),u^{(k)}(t)\bigr).
\end{equation}

\begin{theorem}[Energy dissipation]
\label{thm:ensemble-energy-dissipation}
Suppose Assumptions~\ref{ass:general-first-variation} and
\ref{ass:mirror-map} hold. Let $\{u^{(k)}\}_{k\ge0}$ be generated by
Algorithm~\ref{alg:ensemble-mirror-descent} with $\lambda>0$ and
$\lambda\ge L_{\mathrm{rel}}$. Then
\begin{equation}
\label{eq:ensemble-energy-dissipation}
J^\tau[u^{(k)}]-J^\tau[u^{(k+1)}]
\ge(\lambda-L_{\mathrm{rel}})
\mathcal D_h[u^{(k+1)}\mid u^{(k)}]\ge0.
\end{equation}
If $\lambda>L_{\rm rel}$, then
\begin{equation}
\label{eq:successive-bregman-vanish}
\mathcal D_h[u^{(k+1)}\mid u^{(k)}]\longrightarrow0.
\end{equation}
\end{theorem}

\begin{proof}
Apply relative smoothness \eqref{eq:ensemble-relative-smoothness-H} with $u=u^{(k)}$ and $v=u^{(k+1)}$:
\[
J^\tau[u^{(k+1)}]-J^\tau[u^{(k)}]
\le-\int_0^T\langle Q^{(k)}(t),u^{(k+1)}(t)-u^{(k)}(t)\rangle\,\dd t
+L_{\mathrm{rel}}\mathcal D_h[u^{(k+1)}\mid u^{(k)}].
\]
Comparison with $\theta=u^{(k)}(t)$ in the mirror step \eqref{eq:mirror-update-general} gives, for
a.e.\ $t$,
\[
\langle Q^{(k)}(t),u^{(k+1)}(t)-u^{(k)}(t)\rangle
\ge\lambda D_h(u^{(k+1)}(t)\mid u^{(k)}(t)).
\]
Integrating and substituting proves \eqref{eq:ensemble-energy-dissipation}.
Since $J^\tau$ is uniformly bounded below under Assumption~\ref{ass:general-first-variation},
summing this inequality yields
\[
(\lambda-L_{\mathrm{rel}})\sum_{k=0}^\infty
\mathcal D_h[u^{(k+1)}\mid u^{(k)}]
\le J^\tau[u^{(0)}]-\inf_{u\in\mathcal U[0,T]}J^\tau[u]<\infty,
\]
which proves \eqref{eq:successive-bregman-vanish}.
\end{proof}

\subsection{Convergence rates}
\label{subsec:ensemble-convergence-rates}

For $\tau\ge\kappa_{\mathrm{stab}}$,
Theorem~\ref{thm:ensemble-relative-smoothness} supplies relative
smoothness and relative convexity with the nonnegative modulus
$\mu_{\mathrm{stab}}=\tau-\kappa_{\mathrm{stab}}$.
We combine these bounds using the three-point inequality for the
pointwise mirror step.

\begin{lemma}[Three-point inequality]
\label{lem:ensemble-three-point}
Let $\theta,\vartheta,\zeta\in U$, $q\in\mathbb R^m$, and
$\lambda>0$. If
\[
\vartheta\in\argmax_{\widetilde\theta\in U}
\{q\cdot(\widetilde\theta-\theta)-\lambda D_h(\widetilde\theta\mid\theta)\},
\]
then
\begin{equation}
\label{eq:ensemble-three-point}
q\cdot(\zeta-\theta)-\lambda D_h(\zeta\mid\theta)
\le q\cdot(\vartheta-\theta)-\lambda D_h(\vartheta\mid\theta)
-\lambda D_h(\zeta\mid\vartheta).
\end{equation}
\end{lemma}

\begin{proof}
The optimality condition for $\vartheta$ gives
\[
\bigl(q-\lambda\nabla h(\vartheta)+\lambda\nabla h(\theta)\bigr)
\cdot(\zeta-\vartheta)\le0.
\]
Combine this with the Bregman three-point identity
\[
D_h(\zeta\mid\theta)-D_h(\vartheta\mid\theta)-D_h(\zeta\mid\vartheta)
=(\nabla h(\vartheta)-\nabla h(\theta))\cdot(\zeta-\vartheta)
\]
to obtain \eqref{eq:ensemble-three-point}.
\end{proof}

\begin{theorem}[Convergence rates for ensemble mirror descent]
\label{thm:ensemble-mirror-rates}
Suppose Assumptions~\ref{ass:general-first-variation} and
\ref{ass:mirror-map} hold, and let
$\tau\ge\kappa_{\mathrm{stab}}$. Set
$L_{\mathrm{rel}}=\tau+\kappa_{\mathrm{stab}}$ and
$\mu_{\mathrm{stab}}=\tau-\kappa_{\mathrm{stab}}\ge0$ as in
Theorem~\ref{thm:ensemble-relative-smoothness}, which supplies the
relative-smoothness bound and, for all admissible $u,v$,
\begin{equation}
\label{eq:rate-relative-convexity}
J^\tau[v]-J^\tau[u]\ge dJ^\tau[u](v-u)
+\mu_{\mathrm{stab}}\mathcal D_h[v\mid u].
\end{equation}
Let $\{u^{(k)}\}_{k\ge0}$ be generated
by Algorithm~\ref{alg:ensemble-mirror-descent} with $\lambda>0$ and
$\lambda\ge L_{\mathrm{rel}}$.
If $\tau>\kappa_{\mathrm{stab}}$, then
$0<\mu_{\mathrm{stab}}\le\tau\le L_{\mathrm{rel}}\le\lambda$.
\begin{enumerate}
\item If $\tau=\kappa_{\mathrm{stab}}$, then
$\mu_{\mathrm{stab}}=0$ and, for every $v\in\mathcal U[0,T]$,
\begin{equation}
\label{eq:ensemble-sublinear-rate}
J^\tau[u^{(k)}]-J^\tau[v]\le\frac\lambda k
\mathcal D_h[v\mid u^{(0)}],\qquad k\ge1.
\end{equation}
By Proposition~\ref{prop:existence-uniqueness}, a minimizer $u^\star$
exists. Taking $v=u^\star$ in \eqref{eq:ensemble-sublinear-rate} gives
\[
0\le J^\tau[u^{(k)}]-J^\tau[u^\star]
\le\frac\lambda k\mathcal D_h[u^\star\mid u^{(0)}],\qquad k\ge1.
\]
This estimate holds for every minimizer when the minimizer is nonunique.
\item If $\tau>\kappa_{\mathrm{stab}}$, then the minimizer $u^\star$ is
unique by Proposition~\ref{prop:existence-uniqueness}, and
\begin{equation}
\label{eq:bregman-geometric-rate}
\mathcal D_h[u^\star\mid u^{(k)}]\le
\left(1-\frac{\mu_{\mathrm{stab}}}{\lambda}\right)^k\mathcal D_h[u^\star\mid u^{(0)}],
\end{equation}
and
\begin{equation}
\label{eq:ensemble-geometric-rate}
0\le J^\tau[u^{(k)}]-J^\tau[u^\star]
\le\lambda\left(1-\frac{\mu_{\mathrm{stab}}}{\lambda}\right)^{k}
\mathcal D_h[u^\star\mid u^{(0)}],\qquad k\ge1.
\end{equation}
\end{enumerate}
\end{theorem}

\begin{proof}
Use $Q^{(k)}$ from \eqref{eq:iterate-Hamiltonian-gradient}.
Apply Lemma~\ref{lem:ensemble-three-point} pointwise with
$\theta=u^{(k)}(t)$, $\vartheta=u^{(k+1)}(t)$,
$\zeta=v(t)$, and $q=Q^{(k)}(t)$. Integration gives
\begin{align}
\label{eq:integrated-three-point}
\int_0^T\langle Q^{(k)}(t),v(t)-u^{(k)}(t)\rangle\,\dd t
\le{}&\int_0^T\langle Q^{(k)}(t),u^{(k+1)}(t)-u^{(k)}(t)\rangle\,\dd t\nonumber\\
&-\lambda\mathcal D_h[u^{(k+1)}\mid u^{(k)}]
-\lambda\mathcal D_h[v\mid u^{(k+1)}]
+\lambda\mathcal D_h[v\mid u^{(k)}].
\end{align}
Combine this with relative smoothness and
$\lambda\ge L_{\mathrm{rel}}$ to obtain
\begin{align}
\label{eq:mirror-rate-intermediate}
J^\tau[u^{(k+1)}]\le{}&J^\tau[u^{(k)}]
-\int_0^T\langle Q^{(k)}(t),v(t)-u^{(k)}(t)\rangle\,\dd t\nonumber\\
&+\lambda\mathcal D_h[v\mid u^{(k)}]
-\lambda\mathcal D_h[v\mid u^{(k+1)}].
\end{align}
The lower bound from Theorem~\ref{thm:ensemble-relative-smoothness} gives
\[
J^\tau[v]-J^\tau[u^{(k)}]\ge
-\int_0^T\langle Q^{(k)}(t),v(t)-u^{(k)}(t)\rangle\,\dd t
+\mu_{\mathrm{stab}}\mathcal D_h[v\mid u^{(k)}].
\]
Substitution yields the fundamental one-step inequality
\begin{equation}
\label{eq:fundamental-mirror-rate}
J^\tau[u^{(k+1)}]-J^\tau[v]
\le(\lambda-\mu_{\mathrm{stab}})\mathcal D_h[v\mid u^{(k)}]
-\lambda\mathcal D_h[v\mid u^{(k+1)}].
\end{equation}
If $\mu_{\mathrm{stab}}=0$, summation from $k=0$ to $N-1$ gives
\[
\sum_{k=0}^{N-1}\bigl(J^\tau[u^{(k+1)}]-J^\tau[v]\bigr)
\le\lambda\mathcal D_h[v\mid u^{(0)}].
\]
By energy dissipation, the left-hand side is at least
$N(J^\tau[u^{(N)}]-J^\tau[v])$, proving
\eqref{eq:ensemble-sublinear-rate}. Choosing any minimizer supplied by
Proposition~\ref{prop:existence-uniqueness} gives the stated objective-gap bound.

If $\mu_{\mathrm{stab}}>0$, set $v=u^\star$, where the unique minimizer
exists by Proposition~\ref{prop:existence-uniqueness}. The objective gap is nonnegative,
so \eqref{eq:fundamental-mirror-rate} implies
\[
\mathcal D_h[u^\star\mid u^{(k+1)}]
\le\left(1-\frac{\mu_{\mathrm{stab}}}{\lambda}\right)\mathcal D_h[u^\star\mid u^{(k)}].
\]
Here $0<\mu_{\mathrm{stab}}\le\lambda$, so the
factor lies in $[0,1)$. Iteration proves
\eqref{eq:bregman-geometric-rate}. Dropping the nonpositive last
term in \eqref{eq:fundamental-mirror-rate} also gives
\[
J^\tau[u^{(k+1)}]-J^\tau[u^\star]
\le(\lambda-\mu_{\mathrm{stab}})\mathcal D_h[u^\star\mid u^{(k)}]
{}=\lambda\left(1-\frac{\mu_{\mathrm{stab}}}{\lambda}\right)\mathcal D_h[u^\star\mid u^{(k)}].
\]
Combining with the distance contraction \eqref{eq:bregman-geometric-rate} gives
$J^\tau[u^{(k+1)}]-J^\tau[u^\star]\le\lambda(1-\mu_{\mathrm{stab}}/\lambda)^{k+1}\mathcal D_h[u^\star\mid u^{(0)}]$;
reindexing proves
\eqref{eq:ensemble-geometric-rate}.
\end{proof}

The next section gives explicit constants and projected-gradient
convergence rates for the SSM specialization.

\section{Applications to SSM training}
\label{sec:ssm-consequences}

We now specialize the general convergence theory of
Section~\ref{sec:ensemble-mirror-convergence} to the normalized SSM
training problem $(\mathbf P_S^\tau)$.  Its dynamics and objective are
\[
b(t,x,\theta,\omega)
=A(t,\omega;\theta)x+B(t,\omega;\theta),
\qquad
J^\tau[u]=J^0[u]+\tau\mathcal R_h[u],
\]
where
\[
J^0[u]=\int_0^T\mathbb E_\omega
|Cx^u(t,\omega)-\mathcal F[\omega](t)|^2\,\dd t,
\qquad h(\theta)=\frac12|\theta|^2.
\]
Thus the parameter $\tau$ is the same as in the general formulation,
and the corresponding Bregman geometry is
\begin{equation}
\label{eq:ssm-euclidean-geometry}
D_h(\vartheta\mid\theta)=\frac12|\vartheta-\theta|^2,
\qquad
\mathcal D_h[v\mid u]=\frac12\|v-u\|_{L^2}^2,
\qquad \sigma_h=1.
\end{equation}
We substitute the SSM derivatives into the general stability and
two-sided relative-curvature estimates, obtaining explicit constants
and projected-gradient convergence rates.

\subsection{Explicit constants and two-sided relative curvature for the SSM objective}
\label{subsec:ssm-relative-smoothness}

The affine growth bound
$|b(t,x,\theta,\omega)|\le M_A|x|+M_B$ and the SSM adjoint
\eqref{eq:adj_first_variation} specialize
Lemma~\ref{lem:uniform-state-adjoint-bounds} with
\begin{equation}
\label{eq:ssm-MX-MP}
\begin{aligned}
M_X&:=\bigl(\|x_0\|_\infty+TM_B\bigr)e^{M_AT},\\
M_P&:=2T\|C\|_{\mathrm{op}}\bigl(\|C\|_{\mathrm{op}}M_X+M_{\mathcal F}\bigr)e^{M_AT}.
\end{aligned}
\end{equation}
Here $\|x_0\|_\infty=\sup_{\omega\in\Omega}|x_0(\omega)|$.
Indeed, one may take $B_x=M_A$,
$F_x=2\|C\|_{\mathrm{op}}(\|C\|_{\mathrm{op}}M_X+M_{\mathcal F})$, and $G_x=0$ in that lemma.
Consequently, every admissible control satisfies
\begin{equation}
\label{eq:ssm-uniform-state-adjoint}
|x^u(t,\omega)|\le M_X,
\qquad |p^u(t,\omega)|\le M_P.
\end{equation}
In this section, $\mathcal X$ and $\mathcal P$ denote the bounded
state and adjoint sets of Section~\ref{sec:ensemble-mirror-convergence}
formed using the constants in \eqref{eq:ssm-MX-MP}.

For the SSM dynamics and running cost,
\[
D_x b=A(t,\omega;\theta),
\qquad
D_\theta b[\zeta]
=D_\theta A(t,\omega;\theta)[\zeta]x
+D_\theta B(t,\omega;\theta)[\zeta],
\]
\[
\nabla_x f=2C^\top(Cx-\mathcal F[\omega](t)),
\qquad \nabla_x^2f=2C^\top C,
\qquad \nabla_\theta f=0.
\]
Together with $g=0$, these derivatives verify
Assumption~\ref{ass:general-first-variation} under
Assumption~\ref{asp:main_asp}.  On $|x|\le M_X$, admissible choices
of the derivative bounds in Section~\ref{sec:ensemble-mirror-convergence}
are
\begin{equation}
\label{eq:ssm-derivative-dictionary}
\begin{aligned}
B_x&=M_A,& B_\theta&=G_AM_X+G_B,\\
L_{xx}^b&=0,& L_{x\theta}^b&=G_A,\\
L_{xx}^f&=2\|C\|_{\mathrm{op}}^2,& L_{x\theta}^f&=0,\\
L_{\theta x}^b&=G_A,& L_{\theta\theta}^b&=H_AM_X+H_B,\\
L_{\theta x}^f&=0,& L_{\theta\theta}^f&=0,\qquad L_g=0.
\end{aligned}
\end{equation}
Hence the general stability constants become
\begin{equation}
\label{eq:KX-ssm}
K_X:=(G_AM_X+G_B)e^{M_AT},
\end{equation}
and
\begin{align}
\label{eq:Cx-Cp-ssm}
C_x:=TK_X, \qquad
C_p:=Te^{M_AT}\bigl(M_PG_A+2\|C\|_{\mathrm{op}}^2TK_X\bigr).
\end{align}

\begin{lemma}[State-adjoint stability for SSMs]
\label{lem:ssm-state-adjoint-stability}
Suppose that Assumption~\ref{asp:main_asp} holds.  Then, for all
$u,v\in\mathcal U[0,T]$,
\begin{align}
\label{eq:ssm-state-stability}
\|x^v-x^u\|_{L^2}&\le C_x\|v-u\|_{L^2},\\
\label{eq:ssm-adjoint-stability}
\|p^v-p^u\|_{L^2}&\le C_p\|v-u\|_{L^2}.
\end{align}
\end{lemma}

\begin{proof}
Substitution of \eqref{eq:ssm-derivative-dictionary} into
Lemma~\ref{lem:ensemble-state-adjoint-stability} gives
$A_x=2\|C\|_{\mathrm{op}}^2$, $A_\theta=M_PG_A$, and $L_g=0$, and therefore
exactly the constants in \eqref{eq:KX-ssm}--\eqref{eq:Cx-Cp-ssm}.
\end{proof}

To specialize the Hamiltonian-gradient estimate, recall that the unregularized SSM Hamiltonian satisfies
\begin{equation}
\label{eq:ssm-H0}
H_t^0(\xi,\pi,\theta)
=\mathbb E_\omega\!\left[
\pi(\omega)\cdot\bigl(A(t,\omega;\theta)\xi(\omega)
+B(t,\omega;\theta)\bigr)
-|C\xi(\omega)-\mathcal F[\omega](t)|^2
\right],
\end{equation}
and, for $\zeta\in\mathbb R^m$,
\begin{equation}
\label{eq:ssm-DuH0}
\left\langle\nabla_\theta H_t^0(\xi,\pi,\theta),\zeta\right\rangle
=\mathbb E_\omega\!\left[
\pi(\omega)\cdot\bigl(
D_\theta A(t,\omega;\theta)[\zeta]\xi(\omega)
+D_\theta B(t,\omega;\theta)[\zeta]\bigr)
\right].
\end{equation}
The derivative dictionary gives
\begin{equation}
\label{eq:CH-ssm}
C_H^{\mathrm{SSM}}
:=\max\left\{
M_PG_A,\;G_AM_X+G_B,\;M_P(H_AM_X+H_B)
\right\}.
\end{equation}
Lemma~\ref{lem:ensemble-Hamiltonian-gradient-Lipschitz} therefore yields
\begin{align}
\label{eq:ssm-H-gradient-Lipschitz}
&\left|\nabla_\theta H_t^0(\xi,\pi,\theta)
-\nabla_\theta H_t^0(\widetilde\xi,\widetilde\pi,\widetilde\theta)
\right|\nonumber\\
&\qquad\le C_H^{\mathrm{SSM}}\left(
\|\xi-\widetilde\xi\|_{L^2(\Omega)}
+\|\pi-\widetilde\pi\|_{L^2(\Omega)}
+|\theta-\widetilde\theta|\right)
\end{align}
for $\xi,\widetilde\xi\in\mathcal X$,
$\pi,\widetilde\pi\in\mathcal P$, and
$\theta,\widetilde\theta\in U$.  The three entries in
\eqref{eq:CH-ssm} bound respectively state, adjoint, and parameter
variation.  Define
\begin{equation}
\label{eq:kappa-stab-ssm}
\kappa_{\mathrm{stab}}^{\mathrm{SSM}}
:=C_H^{\mathrm{SSM}}(C_x+C_p+1),
\end{equation}
and hence
\begin{equation}
\label{eq:L-ssm}
L_{\mathrm{SSM}}
:=\tau+\kappa_{\mathrm{stab}}^{\mathrm{SSM}}
=\tau+C_H^{\mathrm{SSM}}(C_x+C_p+1).
\end{equation}

\begin{proposition}[Two-sided relative curvature of the SSM objective]
\label{prop:ssm-relative-smoothness}
Under Assumption~\ref{asp:main_asp}, the SSM training objective
satisfies
\begin{equation}
\label{eq:ssm-relative-smoothness}
\begin{aligned}
dJ^\tau[u](v-u)
+\bigl(\tau-\kappa_{\mathrm{stab}}^{\mathrm{SSM}}\bigr)
\mathcal D_h[v\mid u]
&\le J^\tau[v]-J^\tau[u]\\
&\le dJ^\tau[u](v-u)
+L_{\mathrm{SSM}}\mathcal D_h[v\mid u]
\end{aligned}
\end{equation}
for every $u,v\in\mathcal U[0,T]$, where
$h(\theta)=|\theta|^2/2$ and
$\mathcal D_h[v\mid u]=\|v-u\|_{L^2}^2/2$.
In particular, the objective is relatively convex when
$\tau=\kappa_{\mathrm{stab}}^{\mathrm{SSM}}$ and relatively
strongly convex when $\tau>\kappa_{\mathrm{stab}}^{\mathrm{SSM}}$,
with modulus
\begin{equation}
\label{eq:mu-stab-ssm}
\mu_{\mathrm{stab}}^{\mathrm{SSM}}
:=\tau-\kappa_{\mathrm{stab}}^{\mathrm{SSM}}.
\end{equation}
These conclusions require no rank or injectivity assumption on $C$.
\end{proposition}

\begin{proof}
The SSM derivative bounds in \eqref{eq:ssm-derivative-dictionary} verify
the general estimates with the constants
\eqref{eq:Cx-Cp-ssm} and \eqref{eq:CH-ssm}.  By
\eqref{eq:ssm-euclidean-geometry}, $\sigma_h=1$ and
$\mathcal D_h[v\mid u]=\|v-u\|_{L^2}^2/2$.
Theorem~\ref{thm:ensemble-relative-smoothness} now applies with
$\kappa_{\mathrm{stab}}=\kappa_{\mathrm{stab}}^{\mathrm{SSM}}$,
giving both bounds.  Its constant \eqref{eq:Lrel-general} is
exactly $L_{\mathrm{SSM}}$ in \eqref{eq:L-ssm}.
\end{proof}

\subsection{Projected-gradient convergence for the SSM parameter-path objective}
\label{subsec:ssm-gradient-convergence}

Recall from Remark~\ref{rem:euclidean-mirror-ssm} that, in the
quadratic geometry, Algorithm~\ref{alg:ensemble-mirror-descent}
reduces to
\begin{equation}
\label{eq:ssm-gradient-final}
u^{(k+1)}(t)
=P_U\!\left(
u^{(k)}(t)+\frac1\lambda
\nabla_\theta H_t^\tau\bigl(
x^{u^{(k)}}(t,\cdot),p^{u^{(k)}}(t,\cdot),u^{(k)}(t)
\bigr)\right).
\end{equation}

\begin{theorem}[Projected-gradient convergence for the SSM objective]
\label{thm:ssm-gradient-convergence}
Suppose that Assumption~\ref{asp:main_asp} holds and that
$\{u^{(k)}\}_{k\ge0}$ is generated by
\eqref{eq:ssm-gradient-final} with $\lambda>0$ and
\begin{equation}
\label{eq:ssm-step-condition}
\lambda\ge L_{\mathrm{SSM}},
\end{equation}
where $L_{\mathrm{SSM}}$ is defined in \eqref{eq:L-ssm}.
Then the SSM parameter-path objective is nonincreasing and satisfies
\begin{equation}
\label{eq:ssm-main-dissipation}
J^\tau[u^{(k)}]-J^\tau[u^{(k+1)}]
\ge\frac{\lambda-L_{\mathrm{SSM}}}{2}
\|u^{(k+1)}-u^{(k)}\|_{L^2}^2.
\end{equation}

\begin{enumerate}
\item If
\begin{equation}
\label{eq:ssm-critical-tau}
\tau=\kappa_{\mathrm{stab}}^{\mathrm{SSM}},
\end{equation}
then $(\mathbf P_S^\tau)$ admits a minimizer $u^\star$ and
\begin{equation}
\label{eq:ssm-O1k}
0\le J^\tau[u^{(k)}]-J^\tau[u^\star]
\le\frac{\lambda}{2k}\|u^\star-u^{(0)}\|_{L^2}^2,
\qquad k\ge1.
\end{equation}
\item If
\begin{equation}
\label{eq:ssm-tau-threshold}
\boxed{\tau>\kappa_{\mathrm{stab}}^{\mathrm{SSM}}},
\end{equation}
then the minimizer $u^\star$ is unique,
$\mu_{\mathrm{stab}}^{\mathrm{SSM}}
=\tau-\kappa_{\mathrm{stab}}^{\mathrm{SSM}}>0$, and
\begin{equation}
\label{eq:ssm-control-geometric}
\|u^{(k)}-u^\star\|_{L^2}^2
\le\left(1-\frac{\mu_{\mathrm{stab}}^{\mathrm{SSM}}}{\lambda}\right)^k
\|u^{(0)}-u^\star\|_{L^2}^2.
\end{equation}
Moreover,
\begin{equation}
\label{eq:ssm-objective-geometric}
0\le J^\tau[u^{(k)}]-J^\tau[u^\star]
\le\frac{\lambda}{2}
\left(1-\frac{\mu_{\mathrm{stab}}^{\mathrm{SSM}}}{\lambda}\right)^{k}
\|u^{(0)}-u^{\star}\|_{L^2}^2,
\qquad k\ge1.
\end{equation}
\end{enumerate}
In the strong case,
$0<\mu_{\mathrm{stab}}^{\mathrm{SSM}}\le\tau
\le L_{\mathrm{SSM}}\le\lambda$, so the contraction factor belongs
to $[0,1)$.
\end{theorem}

\begin{proof}
Proposition~\ref{prop:ssm-relative-smoothness} verifies the relative
smoothness hypothesis of
Theorem~\ref{thm:ensemble-energy-dissipation}, yielding
\eqref{eq:ssm-main-dissipation}.
The lower bound in Proposition~\ref{prop:ssm-relative-smoothness}
gives relative convexity with modulus
$\mu_{\mathrm{stab}}^{\mathrm{SSM}}
=\tau-\kappa_{\mathrm{stab}}^{\mathrm{SSM}}$.
The proof of Proposition~\ref{prop:existence-uniqueness}, applied with
these SSM constants, therefore gives existence when
$\tau\ge\kappa_{\mathrm{stab}}^{\mathrm{SSM}}$ and uniqueness when
$\tau>\kappa_{\mathrm{stab}}^{\mathrm{SSM}}$.
In the critical and strong cases the rates follow from
Theorem~\ref{thm:ensemble-mirror-rates} and
$\mathcal D_h[v\mid u]=\|v-u\|_{L^2}^2/2$.
\end{proof}

\begin{remark}[Role of the regularization parameter]
\label{rem:ssm-tau-role}
The threshold
$\tau\ge\kappa_{\mathrm{stab}}^{\mathrm{SSM}}$ ensures relative
convexity and, by Proposition~\ref{prop:existence-uniqueness}, existence
of an optimal parameter path. At the critical value
$\tau=\kappa_{\mathrm{stab}}^{\mathrm{SSM}}$, uniqueness need not hold
but the objective gap decays as $O(1/k)$. Above the threshold the
objective is relatively strongly convex, the minimizer is unique, and
the projected-gradient iterates converge geometrically under the
stepsize condition \eqref{eq:ssm-step-condition}.
\end{remark}

\appendix

\section{Technical regularity and measurability proofs}
\label{app:regularity-measurability}

\subsection{Proof of Lemma~\ref{lem:uniform-state-adjoint-bounds}}
\label{app:regularity-state-adjoint}

\begin{proof}
We first record the measurability needed for the state equation.
For each fixed $(x,\theta)$, Assumption~\ref{ass:general-first-variation}(3)
gives that
\[
(t,\omega)\longmapsto b(t,x,\theta,\omega)
\]
is measurable, while for each fixed $(t,\omega)$,
Assumption~\ref{ass:general-first-variation}(4) gives continuity of
$(x,\theta)\mapsto b(t,x,\theta,\omega)$. Since
$\mathbb R^n\times U$ is a separable metric space, the
Carath\'eodory joint-measurability theorem implies that
\[
(t,\omega,x,\theta)
\longmapsto b(t,x,\theta,\omega)
\]
is jointly measurable. Because $t \mapsto u(t)$ is measurable, composition with
\[
(t,\omega,x)\longmapsto (t,\omega,x,u(t))
\]
shows that
\begin{equation}
\label{eq:controlled-drift-joint-measurable}
(t,\omega,x)\longmapsto b(t,x,u(t),\omega)
\end{equation}
is jointly measurable.

Fix $\omega$. The map
\[
(t,x)\longmapsto b(t,x,u(t),\omega)
\]
therefore satisfies the Carath\'eodory conditions. Moreover, on every
bounded $x$-region, the locally uniform bound on $D_x b$ and the
mean-value theorem give a uniform local Lipschitz bound in $x$.
Hence the state equation has a unique local absolutely continuous
solution.

For the affine growth bound
\[
|b(t,x,\theta,\omega)|\le c_0+c_1|x|
\]
(including $c_0=c_1=M_b$), the state integral equation gives
\[
|x^u(t,\omega)|
\le
\|x_0\|_\infty+c_0T
+c_1\int_0^t |x^u(s,\omega)|\,\dd s.
\]
Gronwall's inequality yields
\[
|x^u(t,\omega)|
\le
(\|x_0\|_\infty+c_0T)e^{c_1T}.
\]
Thus finite-time escape cannot occur, so the unique local solution
extends to all of $[0,T]$. In particular, with $c_0=c_1=M_b$ this
gives \eqref{eq:uniform-state-bound-general}.

We next establish joint measurability of the state.  Set
\[
\widehat b(t,\omega,x)
:=
b(t,x,u(t),\omega),
\]
which is jointly measurable in $(t,\omega,x)$ by
\eqref{eq:controlled-drift-joint-measurable}.

Choose
\[
R:=M_X+1,
\]
and let
\[
L_R
:=
\sup_{\substack{t\in[0,T],\,\omega\in\Omega\\
                 |x|\le R,\,\theta\in U}}
\|D_x b(t,x,\theta,\omega)\|_{\mathrm{op}}
<\infty.
\]
For the affine growth bound write
\[
K_R:=c_0+c_1R.
\]
Choose $\delta>0$ sufficiently small that
\[
\delta L_R<1,
\qquad
\delta K_R\le 1.
\]
These choices are independent of $\omega$.

We first work on $I_0=[0,\delta\wedge T]$.  Define the Picard
iterates by
\[
x^{(0)}_0(t,\omega):=x_0(\omega),
\]
and, recursively,
\begin{equation}
\label{eq:picard-state-first-interval}
x^{(k+1)}_0(t,\omega)
:=
x_0(\omega)
+
\int_0^t
\widehat b\bigl(s,\omega,x^{(k)}_0(s,\omega)\bigr)\,\dd s.
\end{equation}
We claim inductively that every $x^{(k)}_0$ is jointly measurable in
$(t,\omega)$.  This is clear for $k=0$.  If $x^{(k)}_0$ is jointly
measurable, then
\[
(s,\omega)
\longmapsto
\bigl(s,\omega,x^{(k)}_0(s,\omega)\bigr)
\]
is measurable, and therefore, by joint measurability of $\widehat b$,
\[
q_k(s,\omega)
:=
\widehat b\bigl(s,\omega,x^{(k)}_0(s,\omega)\bigr)
\]
is jointly measurable in $(s,\omega)$.  
Moreover, since $|x^{(k)}_0(s,\omega)|\le R$ on the Picard ball,
the affine growth bound gives
\[
|q_k(s,\omega)|
\le c_0+c_1R
=K_R.
\]
Hence, for every $(t,\omega)$,
\[
\int_0^{\delta\wedge T}
\mathbf 1_{\{s\le t\}}
|q_k(s,\omega)|\,\dd s
\le K_R(\delta\wedge T)<\infty.
\]
Therefore
\[
(t,\omega)\longmapsto
\int_0^t q_k(s,\omega)\,\dd s
=
\int_0^{\delta\wedge T}
\mathbf 1_{\{s\le t\}}q_k(s,\omega)\,\dd s
\]
is jointly measurable. Hence $x^{(k+1)}_0$ is jointly measurable.

For each fixed $\omega$, the Picard map associated with
\eqref{eq:picard-state-first-interval} maps the sup-norm ball
\[
\left\{
z\in C(I_0;\mathbb R^n):
\sup_{t\in I_0}|z(t)-x_0(\omega)|\le1
\right\}
\]
into itself, because $|x_0(\omega)|\le M_X$ and
$\delta K_R\le1$.  Moreover, for two functions $z,\widetilde z$
in this ball,
\[
\sup_{t\in I_0}
\left|
\int_0^t
\bigl(
\widehat b(s,\omega,z(s))
-
\widehat b(s,\omega,\widetilde z(s))
\bigr)\,\dd s
\right|
\le
\delta L_R
\sup_{t\in I_0}|z(t)-\widetilde z(t)|.
\]
Since $\delta L_R<1$, the Picard map is a contraction.  Thus
$x^{(k)}_0(\cdot,\omega)$ converges uniformly on $I_0$ to the unique
state solution $x^u(\cdot,\omega)$.  In particular,
\[
x^u(t,\omega)
=
\lim_{k\to\infty}x^{(k)}_0(t,\omega),
\qquad
(t,\omega)\in I_0\times\Omega,
\]
and hence $x^u$ is jointly measurable on $I_0$ as a pointwise limit
of jointly measurable functions.

The same argument can now be iterated on successive intervals.
Indeed, suppose joint measurability has been established up to some
$t_j<T$.  Then
\[
\omega\longmapsto x^u(t_j,\omega)
\]
is measurable and satisfies
\[
|x^u(t_j,\omega)|\le M_X.
\]
On $I_j:=[t_j,(t_j+\delta)\wedge T]$, define
\[
x^{(0)}_j(t,\omega):=x^u(t_j,\omega),
\]
and
\begin{equation}
\label{eq:picard-state-successive-interval}
x^{(k+1)}_j(t,\omega)
:=
x^u(t_j,\omega)
+
\int_{t_j}^t
\widehat b\bigl(s,\omega,x^{(k)}_j(s,\omega)\bigr)\,\dd s.
\end{equation}
Exactly the same induction shows that every $x^{(k)}_j$ is jointly
measurable on $I_j\times\Omega$, while the same bounds
$\delta K_R\le1$ and $\delta L_R<1$ give convergence to the unique
state solution there.  Since finitely many such intervals cover
$[0,T]$, it follows that
\[
(t,\omega)\longmapsto x^u(t,\omega)
\]
is jointly measurable on $[0,T]\times\Omega$.

We now turn to the coefficients of the adjoint equation. Fix $x \in \mathbb{R}^n$ and $\theta \in U$.  For each coordinate direction $e_j \in \mathbb{R}^n$, fix any sequence of nonzero $\varepsilon_\ell \to 0$.  By Assumption~\ref{ass:general-first-variation}(3), for every $\ell$,
\[
(t,\omega)\longmapsto
\frac{b(t,x+\varepsilon_\ell e_j,\theta,\omega)-b(t,x,\theta,\omega)}{\varepsilon_\ell},
\]
and
\[(t,\omega)\longmapsto
\frac{f(t,x+\varepsilon_\ell e_j,\theta,\omega)-f(t,x,\theta,\omega)}{\varepsilon_\ell}
\]
are measurable. Assumption~\ref{ass:general-first-variation}(4) gives differentiability in $x$, and hence
\[\begin{aligned}
\partial_{x_j}b(t,x,\theta,\omega)
&=\lim_{\ell\to\infty}
\frac{b(t,x+\varepsilon_\ell e_j,\theta,\omega)-b(t,x,\theta,\omega)}{\varepsilon_\ell},\\
\partial_{x_j}f(t,x,\theta,\omega)
&=\lim_{\ell\to\infty}
\frac{f(t,x+\varepsilon_\ell e_j,\theta,\omega)-f(t,x,\theta,\omega)}{\varepsilon_\ell}.
\end{aligned}\]
Thus each coordinate derivative is a pointwise limit of measurable functions, and hence measurable.  By Assumption~\ref{ass:general-first-variation}(4), these derivatives are continuous in $(x,\theta)$ for each fixed $(t,\omega)$. 
Applying the Carath\'eodory joint-measurability theorem once more gives joint measurability of
\[
(t,\omega,x,\theta)
\longmapsto D_x b(t,x,\theta,\omega),
\qquad
(t,\omega,x,\theta)
\longmapsto \nabla_x f(t,x,\theta,\omega).
\]
Composition with the jointly measurable state $x^u$ and the measurable
control $u$ proves \eqref{eq:state-adjoint-coefficient-measurability}.

Likewise, Assumption~\ref{ass:general-first-variation}(5) gives
measurability of
\[
\omega\longmapsto \nabla_x g(x,\omega)
\]
for each fixed $x$, while $x\mapsto\nabla_xg(x,\omega)$ is continuous
for each fixed $\omega$. The same Carath\'eodory argument therefore
shows that $(x,\omega)\mapsto\nabla_xg(x,\omega)$ is jointly
measurable. Since $\omega\mapsto x^u(T,\omega)$ is measurable,
\eqref{eq:terminal-gradient-measurability} follows.

Finally, along the bounded state trajectory the adjoint equation is a
linear backward equation whose coefficient matrix is bounded by
$B_x$, whose forcing is bounded by $F_x$, and whose terminal value is
bounded by $G_x$. Its integral equation gives
\[
|p^u(t,\omega)|
\le
G_x+TF_x
+B_x\int_t^T |p^u(s,\omega)|\,\dd s.
\]
Backward Gronwall's inequality yields
\[
|p^u(t,\omega)|
\le
(G_x+TF_x)e^{B_xT},
\]
which proves \eqref{eq:uniform-adjoint-bound-general}.
Uniqueness follows from the same backward Gronwall argument
applied to the difference of two solutions: it solves the homogeneous
linear equation with zero terminal value, and is therefore identically zero.
Since the coefficient, forcing, and terminal datum are measurable by
\eqref{eq:state-adjoint-coefficient-measurability}--%
\eqref{eq:terminal-gradient-measurability}, applying the same Picard
argument to the time-reversed linear adjoint equation gives a jointly
measurable adjoint $p^u$. All bounds are independent of $u$.
\end{proof}

\subsection{Proof of Lemma~\ref{lem:parameter-derivative-measurability}}
\label{app:regularity-parameter-derivatives}

\begin{proof}
Fix $x\in\mathbb R^n$ and
$\theta\in\operatorname{int}U$.
For each coordinate direction $e_j\in\mathbb R^m$, choose a sequence
of nonzero $\varepsilon_\ell\to0$ such that
\[
\theta+\varepsilon_\ell e_j\in U.
\]
By Assumption~\ref{ass:general-first-variation}(3), for every $\ell$,
\[
(t,\omega)
\longmapsto
\frac{b(t,x,\theta+\varepsilon_\ell e_j,\omega)-b(t,x,\theta,\omega)}{\varepsilon_\ell},
\]
and
\[
(t,\omega)
\longmapsto
\frac{f(t,x,\theta+\varepsilon_\ell e_j,\omega)-f(t,x,\theta,\omega)}{\varepsilon_\ell}
\]
are measurable. Assumption~\ref{ass:general-first-variation}(4)
gives differentiability in $\theta$, and hence
\[
\begin{aligned}
\partial_{\theta_j}b(t,x,\theta,\omega)
&=\lim_{\ell\to\infty}
\frac{b(t,x,\theta+\varepsilon_\ell e_j,\omega)-b(t,x,\theta,\omega)}{\varepsilon_\ell},\\
\partial_{\theta_j}f(t,x,\theta,\omega)
&=\lim_{\ell\to\infty}
\frac{f(t,x,\theta+\varepsilon_\ell e_j,\omega)-f(t,x,\theta,\omega)}{\varepsilon_\ell}.
\end{aligned}
\]
Thus each coordinate derivative is measurable in $(t,\omega)$ as a
pointwise limit of measurable functions. Therefore $D_\theta b$ and
$\nabla_\theta f$ are measurable in $(t,\omega)$ for every fixed
$(x,\theta)\in\mathbb R^n\times\operatorname{int}U$.

We next extend this conclusion to $\theta\in\partial U$. Choose a fixed point $\theta^\circ\in\operatorname{int}U$.
For any $\theta\in U$, define
\[
\theta_\ell
:=
\left(1-\frac1\ell\right)\theta
+
\frac1\ell\theta^\circ,
\qquad \ell\ge2.
\]
Convexity of $U$ and
$\theta^\circ\in\operatorname{int}U$ imply
\[
\theta_\ell\in\operatorname{int}U,
\qquad
\theta_\ell\longrightarrow\theta.
\]
For each $\ell$,
\[
(t,\omega)
\longmapsto
D_\theta b(t,x,\theta_\ell,\omega)
\]
is measurable by the preceding argument, while continuity in
$(x,\theta)$ from
Assumption~\ref{ass:general-first-variation}(4) gives, pointwise in
$(t,\omega)$,
\[
D_\theta b(t,x,\theta_\ell,\omega)
\longrightarrow
D_\theta b(t,x,\theta,\omega).
\]
Hence
$(t,\omega)\mapsto D_\theta b(t,x,\theta,\omega)$
is measurable for every fixed $(x,\theta)\in\mathbb R^n\times U$.
The same reasoning applies to $\nabla_\theta f$.

For each fixed $(t,\omega)$,
Assumption~\ref{ass:general-first-variation}(4) also gives continuity
of
\[
(x,\theta)
\longmapsto
D_\theta b(t,x,\theta,\omega),
\qquad
(x,\theta)
\longmapsto
\nabla_\theta f(t,x,\theta,\omega).
\]
Since $\mathbb R^n\times U$ is a separable metric space, the
Carath\'eodory joint-measurability theorem therefore yields joint
measurability of
\[
(t,\omega,x,\theta)
\longmapsto
D_\theta b(t,x,\theta,\omega),
\qquad
(t,\omega,x,\theta)
\longmapsto
\nabla_\theta f(t,x,\theta,\omega).
\]

Now let $\xi$ and $u$ be as in the statement.
Since
\[
(t,\omega)
\longmapsto
\bigl(t,\omega,\xi(t,\omega),u(t)\bigr)
\]
is measurable, composition gives joint measurability of
\[
(t,\omega)
\longmapsto
D_\theta b
\bigl(t,\xi(t,\omega),u(t),\omega\bigr),
\qquad
(t,\omega)
\longmapsto
\nabla_\theta f
\bigl(t,\xi(t,\omega),u(t),\omega\bigr).
\]
\end{proof}

\section{Proof of Lemma~\ref{lem:first-variation-general}}\label{app:proof_first_variation_general}

\begin{proof}
Fix $u,v\in\mathcal U[0,T]$, set $\delta u:=v-u$, and define
\[
u^\varepsilon:=u+\varepsilon\delta u,
\qquad 0\le\varepsilon\le1.
\]
Convexity of $U$ makes this a feasible segment. We compute the
one-sided derivative at $\varepsilon=0$. By
Lemma~\ref{lem:uniform-state-adjoint-bounds}, the states $x^u$ and
$x^{u^\varepsilon}$ and the adjoint $p^u$ are well posed and jointly
measurable, with the uniform bounds $M_X$ and $M_P$ established there.
All state segments below remain in $|\xi|\le M_X$.

Reuse $B_x$ from Lemma~\ref{lem:uniform-state-adjoint-bounds} and set
\[
B_\theta:=\sup_{\substack{t\in[0,T],\,\omega\in\Omega\\
 |x|\le M_X,\,\theta\in U}}
\|D_\theta b(t,x,\theta,\omega)\|_{\mathrm{op}},
\qquad K_X:=B_\theta e^{B_xT}.
\]
The derivative bound is finite by
Assumption~\ref{ass:general-first-variation}.

Abbreviate
\[
\begin{aligned}
b_x^u(t,\omega)&:=D_x b(t,x^u(t,\omega),u(t),\omega),\\
b_\theta^u(t,\omega)&:=D_\theta b(t,x^u(t,\omega),u(t),\omega),\\
f_x^u(t,\omega)&:=\nabla_x f(t,x^u(t,\omega),u(t),\omega),\\
f_\theta^u(t,\omega)&:=\nabla_\theta f(t,x^u(t,\omega),u(t),\omega).
\end{aligned}
\]
By Lemma~\ref{lem:parameter-derivative-measurability},
composition with the jointly measurable state and measurable
control shows that $b_\theta^u$ and $f_\theta^u$ are jointly
measurable in $(t,\omega)$, while $b_x^u$ and $f_x^u$ are
jointly measurable by
Lemma~\ref{lem:uniform-state-adjoint-bounds}.
Introduce the linearized state $y$ as follows. 
\begin{equation}
\label{eq:linearized-state}
\dot y(t,\omega)
=b_x^u(t,\omega)y(t,\omega)
+b_\theta^u(t,\omega)[\delta u(t)],
\qquad y(0,\omega)=0.
\end{equation}
Its bounded measurable coefficients give a unique, jointly measurable
solution, absolutely continuous in time. We shall prove
\begin{equation}
\label{eq:state-diff-quotient}
\sup_{\omega\in\Omega}\sup_{t\in[0,T]}
\left|
\frac{x^{u^\varepsilon}(t,\omega)-x^u(t,\omega)}{\varepsilon}
-y(t,\omega)
\right|\longrightarrow0
\qquad\text{as }\varepsilon\to0^+.
\end{equation}

On the same bounded state-parameter region, define
\begin{equation}
\label{eq:uniform-b-derivative-bound}
B_2:=\sup_{\substack{t\in[0,T],\,\omega\in\Omega\\
 |\xi|\le M_X,\,\theta\in U}}
\|D^2_{(x,\theta)}b(t,\xi,\theta,\omega)\|_{\mathrm{op}}<\infty.
\end{equation}
Also set $M_U:=\max_{\theta\in U}|\theta|$ and
$M_Y:=K_X\|\delta u\|_{L^1}$, where
$\|\delta u\|_{L^1}:=\int_0^T|\delta u(t)|\,\dd t$.
Set $\Delta x^\varepsilon:=x^{u^\varepsilon}-x^u$.
Subtracting the state equations and using the mean-value theorem gives
\[
|\Delta x^\varepsilon(t,\omega)|
\le\int_0^t\bigl(B_x|\Delta x^\varepsilon(s,\omega)|
+\varepsilon B_\theta|\delta u(s)|\bigr)\,\dd s.
\]
Gronwall's inequality therefore yields
\[
|\Delta x^\varepsilon(t,\omega)|
\le\varepsilon B_\theta e^{B_xT}\int_0^t|\delta u(s)|\,\dd s
\le\varepsilon K_X\|\delta u\|_{L^1}.
\]
Consequently,
\begin{equation}
\label{eq:state-perturbation-bound}
\sup_{\omega\in\Omega}\sup_{t\in[0,T]}
|x^{u^\varepsilon}(t,\omega)-x^u(t,\omega)|
\le\varepsilon K_X\|\delta u\|_{L^1}
=\varepsilon M_Y.
\end{equation}
The variation-of-constants formula for \eqref{eq:linearized-state}
gives directly
\[
|y(t,\omega)|\le B_\theta\int_0^t e^{B_x(t-s)}|\delta u(s)|\,\dd s
\le K_X\|\delta u\|_{L^1},
\]
and hence
\begin{equation}
\label{eq:linearized-state-bound}
\sup_{\omega\in\Omega}\sup_{t\in[0,T]}|y(t,\omega)|\le M_Y.
\end{equation}

Write $r^\varepsilon:=\Delta x^\varepsilon-\varepsilon y$.

Taylor's formula along the segment joining
$(x^u(t,\omega),u(t))$ and
$(x^{u^\varepsilon}(t,\omega),u^\varepsilon(t))$ gives
\[
b(t,x^{u^\varepsilon},u^\varepsilon,\omega)
-b(t,x^u,u,\omega)
=b_x^u\Delta x^\varepsilon
+\varepsilon b_\theta^u[\delta u]+\rho^\varepsilon,
\]
where suppressed arguments are $(t,\omega)$ for ensemble quantities
and $t$ for controls. Since $|\delta u(t)|\le2M_U$ a.e.,
\begin{equation}
\label{eq:taylor-remainder-bound}
|\rho^\varepsilon(t,\omega)|
\le\frac{B_2}{2}
\bigl(|\Delta x^\varepsilon(t,\omega)|^2
+\varepsilon^2|\delta u(t)|^2\bigr)
\le\frac{B_2}{2}(M_Y^2+4M_U^2)\varepsilon^2.
\end{equation}
Consequently,
\begin{equation}
\label{eq:remainder-dynamics}
\dot r^\varepsilon(t, \omega)=b_x^u(t, \omega) r^\varepsilon(t, \omega)+\rho^\varepsilon(t, \omega),
\qquad r^\varepsilon(0,\omega)=0,
\end{equation}
and
\begin{equation}
\label{eq:rho-order-two}
\sup_{\omega\in\Omega}\int_0^T
|\rho^\varepsilon(t,\omega)|\,\dd t
\le\frac{TB_2}{2}(M_Y^2+4M_U^2)\varepsilon^2.
\end{equation}
Thus the variation-of-constants formula gives the explicit remainder
estimate
\begin{equation}
\label{eq:remainder-order-two}
\sup_{\omega\in\Omega}\sup_{t\in[0,T]}
|r^\varepsilon(t,\omega)|
\le R_b\varepsilon^2,
\qquad
R_b:=\frac{TB_2e^{B_xT}}{2}(M_Y^2+4M_U^2).
\end{equation}
Dividing \eqref{eq:remainder-order-two} by $\varepsilon$ proves \eqref{eq:state-diff-quotient}.

We next differentiate the cost. Let
\[
F_2:=\sup_{\substack{t\in[0,T],\,\omega\in\Omega\\|\xi|\le M_X,\,\theta\in U}}\|D^2_{(x,\theta)}f(t, \xi, \theta, \omega)\|_{\mathrm{op}},
\qquad H_2:=\sup_{\theta\in U}\|D^2h(\theta)\|_{\mathrm{op}}.
\]
These constants are finite by
Assumption~\ref{ass:general-first-variation}. Taylor's formula gives
\begin{align*}
&f(t,x^{u^\varepsilon},u^\varepsilon,\omega)
-f(t,x^u,u,\omega)
+\tau\bigl(h(u^\varepsilon)-h(u)\bigr)\\
&\qquad=f_x^u\cdot\Delta x^\varepsilon
+\varepsilon f_\theta^u\cdot\delta u
+\varepsilon\tau\nabla h(u)\cdot\delta u
+\widetilde\rho^\varepsilon,
\end{align*}
with
\begin{equation}
\label{eq:L-remainder}
|\widetilde\rho^\varepsilon(t,\omega)|
\le R_f\varepsilon^2,
\qquad
R_f:=\frac{F_2}{2}(M_Y^2+4M_U^2)+2\tau H_2M_U^2.
\end{equation}
Using $\Delta x^\varepsilon=\varepsilon y+r^\varepsilon$ and the
constant $F_x$ from Lemma~\ref{lem:uniform-state-adjoint-bounds},
we obtain
\begin{align}
&\int_0^T\mathbb E_\omega
\bigl[f(t,x^{u^\varepsilon},u^\varepsilon,\omega)
-f(t,x^u,u,\omega)\bigr]\,\dd t
+\tau\bigl(\mathcal R_h[u^\varepsilon]-\mathcal R_h[u]\bigr)
\nonumber\\
&\qquad=\varepsilon\int_0^T
\left\{\mathbb E_\omega
\bigl[f_x^u\cdot y+f_\theta^u\cdot\delta u\bigr]
+\tau\nabla h(u)\cdot\delta u\right\}\,\dd t
+\mathcal E_f^\varepsilon,
\qquad
|\mathcal E_f^\varepsilon|
\le T(F_xR_b+R_f)\varepsilon^2. \nonumber
\end{align}
Therefore,
\begin{align}
\label{eq:running-cost-expansion}
&\frac{1}{\varepsilon} \Bigg( \int_0^T\mathbb E_\omega
\bigl[f(t,x^{u^\varepsilon},u^\varepsilon,\omega)
-f(t,x^u,u,\omega)\bigr]\,\dd t
+\tau\bigl(\mathcal R_h[u^\varepsilon]-\mathcal R_h[u]\bigr) \Bigg)
\nonumber\\
&\qquad \longrightarrow \int_0^T
\left\{\mathbb E_\omega
\bigl[f_x^u\cdot y+f_\theta^u\cdot\delta u\bigr]
+\tau\nabla h(u)\cdot\delta u\right\}\,\dd t.
\end{align}

For the terminal cost, use $G_x$ from
Lemma~\ref{lem:uniform-state-adjoint-bounds} and define locally
\[
L_g:=\sup_{\substack{|x|\le M_X\\\omega\in\Omega}}
\|\nabla_x^2g(x,\omega)\|_{\mathrm{op}}<\infty.
\]
The same argument gives
\begin{align*}
&\mathbb E_\omega
\bigl[g(x^{u^\varepsilon}(T,\omega),\omega)
-g(x^u(T,\omega),\omega)\bigr]\\
&\qquad=\varepsilon\mathbb E_\omega
\bigl[\nabla_xg(x^u(T,\omega),\omega)\cdot y(T,\omega)\bigr]
+\mathcal E_g^\varepsilon,
\qquad
|\mathcal E_g^\varepsilon|
\le\left(G_xR_b+\frac{L_gM_Y^2}{2}\right)\varepsilon^2.
\end{align*}
In particular,
\begin{equation}
\label{eq:terminal-cost-expansion}
\frac{1}{\varepsilon}\,\mathbb E_\omega
\bigl[g(x^{u^\varepsilon}(T,\omega),\omega)
-g(x^u(T,\omega),\omega)\bigr]
\longrightarrow
\mathbb E_\omega
\bigl[\nabla_xg(x^u(T,\omega),\omega)\cdot y(T,\omega)\bigr].
\end{equation}
All expansions have uniform integrable bounds, so they justify
interchanging differentiation with the time and ensemble integrals.
Combining \eqref{eq:running-cost-expansion} and \eqref{eq:terminal-cost-expansion} yields
\begin{align}
\label{eq:pre-adjoint-first-variation}
dJ^\tau[u](\delta u)
={}&\int_0^T\left\{\mathbb E_\omega
\bigl[f_x^u\cdot y+f_\theta^u\cdot\delta u\bigr]
+\tau\nabla h(u)\cdot\delta u\right\}\,\dd t
\nonumber\\
&+\mathbb E_\omega
\bigl[\nabla_xg(x^u(T,\omega),\omega)\cdot y(T,\omega)\bigr].
\end{align}

Finally, the adjoint equation and \eqref{eq:linearized-state} imply,
for the absolutely continuous product $p^u\cdot y$,
\begin{equation}
\label{eq:adjoint-cancellation-pointwise}
\frac{\dd}{\dd t}(p^u\cdot y)
=f_x^u\cdot y+p^u\cdot b_\theta^u[\delta u].
\end{equation}
Integrating and using $y(0,\omega)=0$ and the terminal condition for
$p^u$ gives
\begin{equation}
\label{eq:adjoint-cancellation-integrated}
-\nabla_xg(x^u(T,\omega),\omega)\cdot y(T,\omega)
=\int_0^T\bigl(f_x^u\cdot y+p^u\cdot b_\theta^u[\delta u]\bigr)\,\dd t.
\end{equation}
The bounds from Lemma~\ref{lem:uniform-state-adjoint-bounds} and
\eqref{eq:linearized-state-bound} justify Fubini's theorem, hence
\begin{equation}
\label{eq:adjoint-cancellation-expected}
\mathbb E_\omega
\bigl[\nabla_xg(x^u(T,\omega),\omega)\cdot y(T,\omega)\bigr]
=-\int_0^T\mathbb E_\omega
\bigl[f_x^u\cdot y+p^u\cdot b_\theta^u[\delta u]\bigr]\,\dd t.
\end{equation}
Substitution into \eqref{eq:pre-adjoint-first-variation} eliminates
$y$ and yields
\begin{equation}
\label{eq:first-variation-before-Hamiltonian}
dJ^\tau[u](\delta u)
=\int_0^T\left\{\mathbb E_\omega
\bigl[-p^u\cdot b_\theta^u[\delta u]+f_\theta^u\cdot\delta u\bigr]
+\tau\nabla h(u)\cdot\delta u\right\}\,\dd t.
\end{equation}

By Lemma~\ref{lem:Hamiltonian-parameter-derivative},
the integrand in \eqref{eq:first-variation-before-Hamiltonian}
is exactly
\[
-\left\langle
\nabla_\theta H_t^\tau
\bigl(x^u(t,\cdot),p^u(t,\cdot),u(t)\bigr),
\delta u(t)
\right\rangle.
\]
Therefore
\[
dJ^\tau[u](v-u)
=
-\int_0^T
\left\langle
\nabla_\theta H_t^\tau
\bigl(x^u(t,\cdot),p^u(t,\cdot),u(t)\bigr),
v(t)-u(t)
\right\rangle
\,\dd t.
\]
\end{proof}

\bibliographystyle{plain}
\bibliography{refs}
\end{document}